\documentclass[10pt]{amsart}
\usepackage{amssymb,latexsym,pstricks}

\topskip  \usepackage{color}
\usepackage{hyperref}
\usepackage{tensor}
\usepackage{enumitem}
\usepackage{mathabx}
\usepackage{comment}
\usepackage{amsfonts,amssymb,amsmath,amsthm,graphicx,latexsym}
\usepackage{mathrsfs,dsfont}

\usepackage{tikz}
\usepackage{bm}
\usepackage{caption}

\newtheorem{theorem}{Theorem}[section]
\newtheorem{proposition}[theorem]{Proposition}%
\newtheorem{lemma}[theorem]{Lemma}%
\newtheorem{corollary}[theorem]{Corollary}

\newtheorem{example}[theorem]{Example}%
\newtheorem{remark}[theorem]{Remark}%

\newtheorem{definition}[theorem]{Definition}

\newcommand{\tp}{\mathbin{\hbox{$\bigcirc$\hbox to 0pt{\hspace{-0.81em}$\scriptstyle\top$\hfil}}}}

\begin{document}
\title[Rook thery and normal ordering in the $q$-deformed Ore algebra]{Rook theory, normal ordering in the $q$-deformed Ore algebra and the polynomial generalization}
\author[M. Schork]{Matthias Schork $^{1}$}
\address{$^1$ Institute for Mathematics, Würzburg University, Emil-Fischer Str. 40, 97074 Würzburg, Germany, matthias.schork@mathematik.uni-wuerzburg.de}

\begin{abstract}
For words in the variables $X$ and $Y$ satisfying the commutation relation of the $q$-deformed generalized Ore algebra, $XY-qYX= \mu I + \nu Y$, we show that the corresponding normal ordering coefficients can be given an interpretation in terms of mixed placements of rooks and files. In particular, the associated $q$-deformed Ore-Stirling and Ore-Lah numbers are treated in detail. We show that the $q$-deformed Ore-Stirling numbers (resp., $q$-deformed Ore-Lah numbers) are given as mixed placement numbers of rooks and files on the staircase board (resp., Lah board). Using this combinatorial interpretation, their recurrence relations are derived. In addition, the normal ordered form of the binomial $(X+Y)^m$ in the $q$-deformed generalized Ore algebra is determined. These considerations are then extended to the $q$-deformed polynomial Weyl algebra generated by $X$ and $Y$ satisfying $XY-qYX=f(Y)$ for some polynomial $f\in \mathbb{C}[Y]$. In particular, associated $q$-deformed polynomial Stirling and Lah numbers are introduced and their properties studied. The normal ordered form of the binomial is also extended to the $q$-deformed polynomial Weyl algebra.

\end{abstract}
\maketitle

\section{Introduction}\label{Introduction}
The {\em $q$-deformed generalized Ore algebra} ${\mathcal O}_{\mu,\nu}(q)$ is defined as the complex unital algebra generated by two variables $X$ and $Y$ satisfying the commutation relation
\begin{equation}\label{CROre}
XY-qYX= \mu I + \nu Y
\end{equation}
for some complex parameters $\mu,\nu$. Letting $\mu=\nu=1$, we obtain the $q$-deformed Ore algebra ${\mathcal O}(q)={\mathcal O}_{1,1}(q)$. Letting instead  $q=1$, we obtain the undeformed generalized Ore algebra ${\mathcal O}_{\mu,\nu}={\mathcal O}_{\mu,\nu}(1)$. Letting furthermore $\mu=-1$ and $\nu=-1$, the defining commutation relation becomes $YX-XY=I+Y$ which coincides with the one of Patrias and Pylyavskyy \cite{PAPY2015,PAPY2018} who coined the name {\em Ore algebra}. They observed that by an appropriate scaling of the generators it suffices to consider this reduced form of the  commutation relation but for us it will be convenient to keep the parameters $\mu,\nu$. Patrias and Pylyavskyy \cite{PAPY2015,PAPY2018} studied this commutation relation in the corresponding theory of {\em dual filtered graphs} which is a K-theoretic analogue of Fomin's {\em dual graded graphs} \cite{SF1994,SF1995} and their predecessors, Stanley’s {\em differential posets} \cite{RS1988}. One way to view this theory is as a study of certain combinatorial representations of the first Weyl algebra ${\mathcal W}$. In fact, many results of \cite{SF1994,SF1995,RS1988} can be viewed as normal ordering results in the Weyl algebra. Starting with \cite{PAPY2015,PAPY2018}, some results on normal ordering in the Ore algebra have been derived recently. For $X$ and $Y$ satisfying $YX-XY=I+Y$, Yeliussizov \cite{DY2019} derived the normal ordered form of $Y^nX^m$ and showed that it involves the Eulerian numbers. Kodsueb and Lytvynov \cite{CKEL2025} considered ordering $(XY)^n$ in the parametrized situation $YX-XY=bI+aY$ studied also in \cite{TMMS2023} where Ore-Stirling numbers were introduced as normal ordering coefficients of $(YX)^n$. Recently, considering the $q$-deformed version, $YX-qXY=I+Y$, normal ordering $(XY)^n$ was studied by Briand \cite{Briand2025}. Note that Levandovskyy et al. \cite{LEKO2011} also considered \eqref{CROre} with respect to normal ordering. 

The Ore algebra ${\mathcal O}_{\mu,\nu}(q)$ can also be considered as a common generalization of the Weyl algebra and the shift algebra. Indeed, for $\nu =0$ and $\mu=1$, one obtains from \eqref{CROre} with $q=1$ the defining relation of the {\em Weyl algebra} ${\mathcal W}={\mathcal O}_{1,0}(1)$,
\begin{equation}\label{CRWeyl}
XY-YX=I,
\end{equation}
while, for $\mu =0$ and $\nu=1$, one obtains from \eqref{CROre} with $q=1$ the defining relation of the {\em shift algebra} ${\mathcal S}={\mathcal O}_{0,1}(1)$,
\begin{equation}\label{CRShift}
XY-YX=Y.
\end{equation}
Letting $\mu =\nu=0$, one obtains the {\em quantum plane} ${\mathcal O}_{0,0}(q)$ where $XY=qYX$.

We say a word $\omega$ in the letters $X$ and $Y$ is in {\it normal ordered form} if all the letters $Y$ stand to the left of all the letters $X$, i.e., $\omega$ is written in the form $\omega=\sum_{j,k}c_{j,k}Y^jX^k$ for some complex coefficients $c_{j,k}$, the {\em normal ordering coefficients}. For example, it is a famous result that one has in the Weyl algebra ${\mathcal W}$ the normal ordering result
 \begin{equation}\label{Scherk}
(YX)^n =\sum_{k=0}^n S(n,k) Y^kX^k,
\end{equation}
where the Stirling numbers of the second kind appear as normal ordering coefficients. The Weyl algebra has a concrete representation in terms of the derivative operator $D$ and the multiplication operator $M_x$, acting on functions of a real variable as $(Df)(x)=f'(x)$ and $(M_xf)(x)=xf(x)$, i.e., $DM_x-M_xD=I$. In this context, $M_x$ is also written as $X$, and \eqref{Scherk} translates into $(XD)^n =\sum_{k=0}^n S(n,k) X^kD^k$. In this form it was already known to Scherk in 1823, and he also considered $(X^rD)^n$ and the corresponding normal ordering coefficients. The generalization to other words, e.g., $(X^pD^q)^n$ was considered already in the early 1930s, see \cite{TMMS2016} for a comprehensive history of early results and \cite{BRLO2020} for some new results. See also the beautiful survey of Blasiak and Flajolet \cite{BLFL2011} for the many ramifications of normal ordering in the Weyl algebra (in the appendix one can find a detailed discussion of Scherk's dissertation from 1823). For the shift algebra  ${\mathcal S}$ one has analogous results, and the normal ordering coefficients of $(YX)^n$ in ${\mathcal S}$ are given by the unsigned Stirling numbers of the first kind $|s(n,k)|$. Normal ordering other words in the shift algebra has also been considered, see, e.g., \cite{MS2022} and the references given therein. The $q$-deformation of the Weyl algebra and the shift algebra has also been considered with respect to normal ordering, see \cite{TMMS2016}. In the undeformed case, one may consider, for $s\in \mathbb{N}_0$, the complex unital algebra $\mathcal{A}_{s;h}$ generated by $X$ and $Y$ satisfying
\begin{equation}\label{CRMonom}
XY-YX=hY^s.
\end{equation}
Clearly, for $s=0$ and $h=1$, one recovers the commutation relation of the Weyl algebra, while for $s=1$ and $h=1$ the one of the shift algebra. The case $s=2$ is also known and describes the {\em Jordan plane}, see \cite{TMMS2016} and the references given therein. For arbitrary $s$ the study of this relation -- and normal ordering words in $X$ and $Y$ satisfying \eqref{CRMonom} -- was begun by Burde \cite{Burde2005} and Varvak \cite{AV2005}. It is natural to introduce {\em generalized Stirling numbers} ${\mathfrak S}_{s;h}(n,k)$ as normal ordering coefficients of $(YX)^n$ in $\mathcal{A}_{s;h}$, see \cite{TMMSMS2011,TMMSMS2012}. In the generalized Ore algebra (where $q=1$) analogues of the Stirling numbers have also been introduced as normal ordering coefficients of $(YX)^n$ \cite{TMMS2023}.

As a common generalization of the $q$-deformed Ore algebra ${\mathcal O}_{\mu,\nu}(q)$ and $\mathcal{A}_{s;h}(q)$ where \eqref{CRMonom} is $q$-deformed to $XY-qYX=hY^s$, one may consider the {\em $q$-deformed polynomial Weyl algebra} $\mathcal{A}_{f}(q)$ which is defined as the complex unital algebra generated by $X$ and $Y$ satisfying
\begin{equation}\label{CRParWeyl}
XY-qYX=f(Y), \,\, f(Y)=\sum_{j=0}^s \alpha_j Y^j,
\end{equation}
where $f$ is a polynomial in $Y$ and $\alpha_j \in \mathbb{C}$. Normal ordering in $\mathcal{A}_{f}(q)$ was started in \cite{TMMS2011} but for $q=1$ Viskov \cite{OVV1996} considered some aspects already in 1996 (and, even earlier, in 1979 Irving \cite{RSI1979} derived some ordering results en route). Benkart et al. \cite{BLO2013,BLO2015,BLO2015a} considered $\mathcal{A}_{f}(1)$ in depth (including some ordering results) and this algebra was considered subsequently by several authors, see the recent survey of Lopes \cite{Lopes2024} on different generalizations of the Weyl algebra and related structures (from a more algebraic perspective). In another direction, normal ordering was studied very recently  by Rubiano and Reyes \cite{ARAR2026} for 3-dimensional skew polynomial rings with 3 generators satisfying 3 commutation relations. 

Let us turn to the connection to rook theory. Rook theory began in 1946 with the publication of \cite{KR1946}. Since then, the original problem of counting the number of possibilities of placing (non-attacking) rooks on certain boards, in particular {\it Ferrers boards}, has been extended by many authors in many different directions. In our context, the first hint for a connection between normal ordering coefficients and rook numbers seems to be by Navon \cite{AMN1973}, but the precise connection was given by Varvak in the fundamental paper \cite{AV2005} (see also Fomin \cite{SF1995}). Briefly, to a word in $X$ and $Y$ satisfying \eqref{CRWeyl}, one associates a Ferrers board, and the normal ordering coefficients are given by the rook numbers of the board. For the particular word $(YX)^n$, the board associated is the staircase board $J_n$, and it is well known that the rook numbers of $J_n$ are given by $S(n,k)$, thereby recovering \eqref{Scherk}. Varvak also showed that normal ordering words in $X$ and $Y$ satisfying \eqref{CRMonom} involves the $s$-rook placements introduced by Goldman and Haglund \cite{JGJH2000}. In particular, choosing $s=0$, the $0$-rook placements correspond to rook placements, and it is equally well known that the file numbers of the staircase board $J_n$ are given by $|s(n,k)|$, thereby recovering the analogue of \eqref{Scherk} in the shift algebra. An extension to the $q$-deformed case of \eqref{CRMonom} was given by Celeste et al. \cite{CCG2017}.
 
In \cite{TMMS2023}, analogues of the Stirling numbers of the second kind were defined in the Ore algebra as normal ordering coefficients of $(YX)^n$, and it was shown these {\em Ore-Stirling numbers} can be expressed in terms of Stirling numbers of the second kind and the unsigned Stirling numbers of the first kind. This was shown in an algebraic fashion, but a combinatorial interpretation was lacking. In the present paper, we will give a combinatorial interpretation for normal ordering arbitrary words in ${\mathcal O}_{\mu,\nu}(q)$ defined by \eqref{CROre}. We also show how this interpretation can be extended to the  $q$-deformed polynomial Weyl algebra $\mathcal{A}_{f}(q)$ defined by \eqref{CRParWeyl}. 

The paper is structured as follows. In Section~\ref{Normal}, we recall rook and file placements and their connection to normal ordering. We then introduce a $q$-statistic on mixed rook and file placements and show that this describes the normal ordering coefficients in the $q$-deformed Ore algebra. As an example, we consider the $q$-deformed Ore-Stirling and Ore-Lah numbers in this context. In Section~\ref{Binomial}, we investigate the binomial formula for $(X+Y)^m$ in the $q$-deformed generalized Ore algebra. All these considerations are extended to the $q$-deformed polynomial Weyl algebra in Section~\ref{Polynomial}. Some conclusions are presented in Section~\ref{Conclusion}.

\section{Normal ordering and mixed placements of rooks and files}\label{Normal}
Let a word $\omega$ in the letters $X$ and $Y$ be given. We associate to it a {\em Ferrers board} $B_{\omega}$ -- independent from the commutation relation $X$ and $Y$ satisfy -- in the following fashion: Start from $(0, 0)$ in $\mathbb{Z}^2$ and represent the letter $X$ as a step to the right and the letter $Y$ as a step up, reading the word from left to right. Then the resulting path outlines a Ferrers board denoted by $B_{\omega}$. For instance, $\omega=(YX)^n$ outlines the {\em staircase board} $J_n$. In Figure~\ref{Board}, the staircase board $J_4$ corresponding to $(YX)^4$ is shown on the left and the Ferrers board corresponding to the word $X^2YXYX^2Y$ on the right (the path outlining the board is drawn fat).

\begin{figure}[h]
\centering
\begin{tikzpicture}

\draw (0,0.5) -- (0,2);
\draw[very thick] (0,0) -- (0,0.5);
\draw (0.5,1) -- (0.5,2);
\draw[very thick] (0.5,0.5) -- (0.5,1);
\draw (1,1.5) -- (1,2);   
\draw[very thick] (1,1) -- (1,1.5);
\draw[very thick] (1.5,1.5) -- (1.5,2);  

\draw (0,2) -- (1.5,2);
\draw[very thick] (1.5,2) -- (2,2);
\draw (0,1.5) -- (1,1.5);
\draw[very thick] (1,1.5) -- (1.5,1.5); 
\draw (0,1) -- (0.5,1);
\draw[very thick] (0.5,1) -- (1,1); 
\draw[very thick] (0,0.5) -- (0.5,0.5);


\draw (4,0.5) -- (4,2);
\draw (4.5,0.5) -- (4.5,2);
\draw[very thick] (5,0.5) -- (5,1);
\draw (5,1) -- (5,2);   
\draw[very thick] (5.5,1) -- (5.5,1.5);
\draw (5.5,1.5) -- (5.5,2); 
\draw (6,1.5) -- (6,2);
\draw[very thick] (6.5,1.5) -- (6.5,2);  

\draw (4,2) -- (6.5,2);
\draw (4,1.5) -- (5.5,1.5);
\draw[very thick] (5.5,1.5) -- (6.5,1.5);
\draw (4,1) -- (5,1);
\draw[very thick] (5,1) -- (5.5,1); 
\draw[very thick] (4,0.5) -- (4.5,0.5);
\draw[very thick] (4.5,0.5) -- (5,0.5);

\end{tikzpicture}
\caption{The Ferrers boards associated to $(YX)^4$ (left) and $X^2YXYX^2Y$ (right).}\label{Board}
\end{figure}
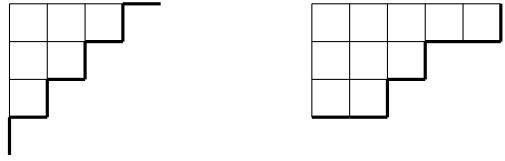
Given a Ferrers board $B$, we call a placement of $k$ rooks in $B$ such that there is at most one rook in each column a {\em file placement of $k$ rooks}, or also a $k$-{\em file placement}. The set of all $k$-file placements on $B$ will be denoted by $\mathcal{F}_k(B)$, and $f_k(B)=|\mathcal{F}_k(B)|$ will be called $k$-{\em th file number of $B$}. A $k$-{\em rook placement} is a special kind of file placement where in addition no two rooks are in the same row (i.e., the $k$ rooks are {\it non-attacking}), see Figure~\ref{FigFile}. The set of all $k$-rook placements on $B$ will be denoted by $\mathcal{R}_k(B)$, and $r_k(B)=|\mathcal{R}_k(B)|$ will be called $k$-{\em th rook number of $B$}.

\begin{figure}[h]
\centering
\begin{tikzpicture}
\draw (0,0.5) -- (1,0.5);
\draw (0.5,0.5) -- (0.5,1.5);
\draw (1,0.5) -- (1,2);  
\draw (1.5,1) -- (1.5,2); 
\draw (2,1.5) -- (2,2);
\draw (2.5,1.5) -- (2.5,2);  
\draw (0,0.5) -- (0,2);
\draw (0.5,1) -- (0.5,2); 
\draw (0,1) -- (1.5,1); 
\draw (0,1.5) -- (2.5,1.5); 
\draw (0,2) -- (2.5,2);
\filldraw (0.25,0.75) circle (3pt);
\filldraw (1.25,1.75) circle (3pt);
\filldraw (1.75,1.75) circle (3pt);
\end{tikzpicture}
\caption{A Ferrers board with a $3$-file placement which is not a $3$-rook placement.}\label{FigFile}
\end{figure}
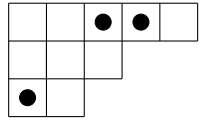
\begin{example}
Let us consider the {\em staircase board} $J_n$. It is well known (see, e.g., \cite[Section 2.4.4]{TMMS2016} and the references given therein) that 
\begin{equation}\label{StairEq}
r_{n-k}(J_n)=S(n,k), \,\,\,\, f_{n-k}(J_n)=|s(n,k)|.
\end{equation}
\end{example}

Now, let us draw the connection to normal ordering. For this, we have to consider the concrete commutation relation between $X$ and $Y$ at hand and use its implications for placing rooks (or files). 

Let us start with the Weyl algebra ${\mathcal W}$ where $XY=YX+\mu I$. Recall that the process of normal ordering a word $\omega$ in ${\mathcal W}$ consists of selecting the rightmost corner of $B_{\omega}$ (i.e., corresponding to the rightmost subword $XY$) and either placing a rook (corresponding to a contraction $XY \rightsquigarrow \mu I$) or leaving it empty (corresponding to the commutation $XY \rightsquigarrow YX $). By iterating over all possibilities of placing the rooks (or not), one obtains that the normal ordering coefficients are given as sum over all possible rook placements. Let $\omega$ be a word containing $m$ letters $X$ and $n$ letters $Y$. Each placement of $k$ rooks on the board -- representing one particular sequence to normal order --  yields a summand in the normal ordered form proportional to $\mu^k Y^{n-k}X^{m-k}$ since $k$ pairs of $X$ and $Y$ are contracted. In total, one obtains the normal ordering result \cite{SF1995,AV2005}
\begin{equation}\label{NORWeyl}
\omega=\sum_{k=0}^{\min(m,n)} \mu^k  r_k(B_{\omega}) Y^{n-k}X^{m-k}.
\end{equation}
For example, when $\omega=(YX)^n$, the corresponding Ferrers board is the staircase board $J_n$, and \eqref{NORWeyl} together with \eqref{StairEq} recovers \eqref{Scherk}. The study of the normal ordering coefficients of arbitrary words in the Weyl algebra was started by McCoy in 1934 \cite{Coy1934} and several different combinatorial interpretations have been given, see \cite{BLFL2011} and \cite{TMMS2016} for a survey.

For the $q$-deformed Weyl algebra, the procedure is similar, but we have to introduce certain $q$-weights for the rook placements to capture the factor $q$ when not placing a rook (corresponding to the $q$-commutation $XY \rightsquigarrow qYX $). For this, let $B$ be a Ferrers board and let $\phi \in \mathcal{R}_k(B)$. We partition the boxes of $B$ as follows into classes:
\begin{enumerate}
\item A box is called a {\em rook box} if a rook is placed in it,
\item A box is called a {\em cancelled box}, if it is not a rook box and lying in the same column above a rook or in the same row to the left of a rook,
\item All remaining boxes are {\em empty boxes}. 
\end{enumerate} 
In the following figures, a rook boxed will be marked by a filled square, a cancelled box by an $X$. 

The {\em $q$-weight of the placement} $\phi \in \mathcal{R}_k(B)$ is defined to be 
\begin{equation}
w_{\mu;q}(B,\phi)=\mu^{\# \mbox{rook boxes}} q^{\# \mbox{empty boxes}}=\mu^{k} q^{\# \mbox{empty boxes}}.
\end{equation}
The {\em $q$-deformed rook numbers} are defined by 
\begin{equation}
r_k(B;q)=\sum_{\phi \in \mathcal{R}_k(B)} w_{\mu;q}(B,\phi),
\end{equation}
and the $q$-deformed analogue of \eqref{NORWeyl} holds true for the $q$-deformed Weyl algebra \cite{AV2005}, 
\begin{equation}\label{NORQWeyl}
\omega=\sum_{k=0}^{\min(m,n)} r_k(B_{\omega};q) Y^{n-k}X^{m-k}.
\end{equation}
Clearly, for $q=1$, one has $r_k(B;1)=\mu^k |\mathcal{R}_k(B)|=\mu^k r_k(B)$, recovering the undeformed situation. The $q$-weights above go back to Garsia and Remmel \cite{GR1986}.
\begin{example}\label{ExampleNOWeyl}
Consider the word $\omega = X^2YXYX^2Y$ in the $q$-deformed Weyl algebra with Ferrers board displayed in Figure~\ref{Board}. Consider the rook placement of two rooks shown in Figure~\ref{RookPl}. Its $q$-weight is given by $ \mu^2 q^3$, thus the contribution of the particular sequence of normal ordering depicted by this placement will be $\mu^2 q^3 YX^{3}$. This rook placement describes the follwowing sequence of steps, from right to left: 1) $q$-commute, 2) contract, 3) $q$-commute, 4) contract, 5) $q$-commute. Thus:
$$
X^2YXYX^2Y \overset{1)}\rightsquigarrow q X^2YXYXYX \overset{2)}\rightsquigarrow \mu q X^2YXYX \overset{3)}\rightsquigarrow \mu q^2 X^2Y^2X^2  \overset{4)}\rightsquigarrow \mu^2 q^2 XYX^2 \overset{5)}\rightsquigarrow \mu^2 q^3 YX^3, 
$$    
as it should.
\begin{figure}[h]
\centering
\begin{tikzpicture}
\draw (0,0.5) -- (1,0.5);
\draw (0.5,0.5) -- (0.5,1.5);
\draw (1,0.5) -- (1,2);  
\draw (1.5,1) -- (1.5,2); 
\draw (2,1.5) -- (2,2);
\draw (2.5,1.5) -- (2.5,2);  
\draw (0,0.5) -- (0,2);
\draw (0.5,1) -- (0.5,2); 
\draw (0,1) -- (1.5,1); 
\draw (0,1.5) -- (2.5,1.5); 
\draw (0,2) -- (2.5,2);
\filldraw (1.625,1.625) rectangle (1.875, 1.875);
\filldraw (0.625,0.625) rectangle (0.875, 0.875);
\node at (0.25,1.75) {$\mbox{X}$};
\node at (0.75,1.75) {$\mbox{X}$};
\node at (1.25,1.75) {$\mbox{X}$};
\node at (0.75,1.25) {$\mbox{X}$};
\node at (0.25,0.75) {$\mbox{X}$};
\end{tikzpicture}
\caption{A rook placement of 2 rooks (boxes are marked according to their class).}\label{RookPl}
\end{figure}
\end{example}

Let us turn to the $q$-deformed shift algebra where $XY=q YX+\nu Y$. The process of normal ordering a word in the $q$-deformed shift algebra consists of selecting the rightmost corner (i.e., corresponding to a subword $XY$) and either placing a file (corresponding to a contraction $XY \rightsquigarrow \nu Y$) or leaving it empty (corresponding to a $q$-commutation $XY \rightsquigarrow qYX $). By iterating over all possibilities of placing the files (or not), one obtains that the normal ordering coefficients are given as sum over all $q$-weighted file placements. Let $B$ be a Ferrers board and let $\phi \in \mathcal{F}_k(B)$. We partition the boxes of $B$ as follows into classes:
\begin{enumerate}
\item A box is called a {\em file box} if a file is paced in it,
\item A box is called a {\em cancelled box}, if it is not a file box and lying in the same column above a file,
\item All remaining boxes are {\em empty boxes}. 
\end{enumerate} 
The {\em $q$-weight of the placement} $\phi \in \mathcal{F}_k(B)$ is defined to be 
\begin{equation}
w_{\nu;q}(B,\phi)=\nu^{\# \mbox{file boxes}} q^{\# \mbox{empty boxes}}=\nu^{k} q^{\# \mbox{empty boxes}}.
\end{equation}
The {\em $q$-deformed file numbers} are defined by 
\begin{equation}
f_k(B;q)=\sum_{\phi \in \mathcal{F}_k(B)} w_{\nu;q}(B,\phi).
\end{equation}
 Let $\omega$ be a word containing $m$ letters $X$ and $n$ letters $Y$. Note that here the number of letters $Y$ is not changed in a contraction. Thus, each placement of $k$ files on the board -- representing one particular sequence to normal order --  yields a summand in the normal ordered form proportional to $\nu^k Y^{n}X^{m-k}$. In total, one has the following result in the $q$-deformed shift algebra, 
\begin{equation}\label{NORQShift}
\omega=\sum_{k=0}^{m} f_k(B_{\omega};q) Y^{n}X^{m-k},
\end{equation}
in analogy to \eqref{NORQWeyl} in the $q$-deformed Weyl algebra. In particular, letting $\omega=(YX)^n$ and $q=1$, one has $B_{\omega}=J_n$ and letting $\ell = n-k$ this yields upon using  \eqref{StairEq}
\begin{equation}\label{ShiftScherk}
(YX)^n=\sum_{\ell=0}^{n} f_{n-\ell}(J_n) Y^{n}X^{\ell}=\sum_{\ell=0}^{n}|s(n,\ell) |Y^{n}X^{\ell},
\end{equation}
the well-know analogue of \eqref{Scherk} in the shift algebra.

Let us turn to the $q$-deformed generalized Ore algebra ${\mathcal O}_{\mu,\nu}(q)$ where $XY=qYX+\mu I + \nu Y$, see \eqref{CROre}. To a word $\omega$ in the letters $X$ and $Y$ we associate a Ferrers board $B_{\omega}$ as described above. The process of normal ordering proceeds exactly as in the case of the $q$-deformed Weyl algebra ${\mathcal O}_{1,0}(q)$ and the $q$-deformed shift algebra ${\mathcal O}_{0,1}(q)$ as described above -- with one important difference: On the right-hand side we now have three options instead of two since here are two ways to contract! Thus, selecting the rightmost corner (i.e., corresponding to a subword $XY$), we can 
\begin{itemize}
\item leave it empty (corresponding to $q$-commute $XY \rightsquigarrow q YX$), or
\item place a rook (corresponding to the strong contraction $XY \rightsquigarrow \mu I$), or
\item place a file (corresponding to the weak contraction $XY \rightsquigarrow \nu Y$).
\end{itemize}
Then one has to sum over all possibilities of placing the rooks or files (or not). We now formalize these notions. 

\begin{definition}
Let a Ferrers board $B$ be given. A placement of $k$ rooks and $\ell$ files on the board $B$ will be called {\em non-attacking}, if 
\begin{itemize}
\item No two (or more) rooks lie in the same row or column, and
\item No two (or more) files lie in the same column, and
\item No file and rook lie in the same column, and no file lies to the left of a rook in the same row.
\end{itemize}
The set of all non-attacking placements of $k$ rooks and $\ell$ files on the board $B$ will be denoted by $\mathcal{M}_{k,\ell}(B)$ ($\mathcal{M}$ for ``mixed''). 
\end{definition}
The first condition is the rule for placing rooks, the second condition is the rule for placing files, and the third conditon is the rule for their ``interaction''. As an example, in Figure~\ref{MixedPl}, a board $B$ is shown with a non-attacking placement of 2 rooks and 3 files. Here and in the rest of this section we mark a rook by a filled square and a file by a filled circle to distinguish them. 

\begin{figure}[h]
\centering
\begin{tikzpicture}
\draw (-0.5,2) -- (3,2);
\draw (-0.5,1.5) -- (3,1.5); 
\draw (-0.5,1) -- (2,1);
\draw (-0.5,0.5) -- (1.5,0.5);
\draw (-0.5,0) -- (1,0);
\draw (-0.5,-0.5) -- (1,-0.5);
\draw (-0.5,-1) -- (0.5,-1);

\draw (-0.5,-1) -- (-0.5,2);
\draw (0,-1) -- (0,2);
\draw (0.5,-1) -- (0.5,2);
\draw (1,-0.5) -- (1,2);
\draw (1.5,0.5) -- (1.5,2); 
\draw (2,1) -- (2,2);
\draw (2.5,1.5) -- (2.5,2); 
\draw (3,1.5) -- (3,2);  

\filldraw (1.25,0.75) circle (3pt);
\filldraw (1.75,1.75) circle (3pt);
\filldraw (2.25,1.75) circle (3pt);
\filldraw (0.125,0.625) rectangle (0.375, 0.875);
\filldraw (0.625,-0.375) rectangle (0.875, -0.125);
\end{tikzpicture}
\caption{A Ferrers board with a non-attacking placement of $2$ rooks and 3 files.}\label{MixedPl}
\end{figure}
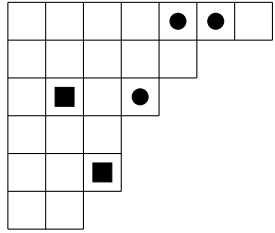

Let $B$ be a Ferrers board and let $\phi \in \mathcal{M}_{k,\ell}(B)$. Similar to above, we partition the boxes of $B$ as follows into classes:
\begin{enumerate}
\item A box is called a {\em rook box} if a rook is placed in it,
\item A box is called a {\em file box} if a file is placed in it,
\item A box is called a {\em cancelled box}, if it is neither a rook box nor a file box, and
\begin{itemize}
\item it is lying above a rook in the same column or to the left of a rook in the same row, or
\item it is lying above a file in the same column.
\end{itemize}
\item All remaining boxes are {\em empty boxes}. 
\end{enumerate}
The {\em $q$-weight of the mixed placement} $\phi \in \mathcal{M}_{k,\ell}(B)$ is defined to be 
\begin{equation}
w_{\mu,\nu;q}(B,\phi)=\mu^{\# \mbox{rook boxes}}\nu^{\# \mbox{file boxes}} q^{\# \mbox{empty boxes}}=\mu^{k}\nu^{\ell} q^{\# \mbox{empty boxes}}.
\end{equation}

\begin{example}\label{ExampleNOOre}
In Figure~\ref{MixedPlClass}, the classes of the boxes are added to Figure~\ref{MixedPl}, implying that the $q$-weight of this mixed placement is given by $q^{8}\mu^2\nu^3$. Similar as in Example~\ref{ExampleNOWeyl}, this placement describes a particular sequence of steps to normal order the word $X^2YXY^2XYXYX^2Y$ (containing 7 letters $X$ and 6 letters $Y$): 1) $q$-commute, 2) contract weakly, 3) $q$-commute, 4) contract weakly, 5) contract weakly, 6) contract strongly, 7) $q$-commute,... Since there are 2 strong contractions and 3 weak contractions, the resulting normal ordered summand is $q^{8}\mu^2\nu^3 Y^{6-2}X^{7-3-2}=q^{8}\mu^2\nu^3 Y^{4}X^{2}$, as can be checked by direct computation. 

\begin{figure}[h]
\centering
\begin{tikzpicture}
\draw (-0.5,2) -- (3,2);
\draw (-0.5,1.5) -- (3,1.5); 
\draw (-0.5,1) -- (2,1);
\draw (-0.5,0.5) -- (1.5,0.5);
\draw (-0.5,0) -- (1,0);
\draw (-0.5,-0.5) -- (1,-0.5);
\draw (-0.5,-1) -- (0.5,-1);

\draw (-0.5,-1) -- (-0.5,2);
\draw (0,-1) -- (0,2);
\draw (0.5,-1) -- (0.5,2);
\draw (1,-0.5) -- (1,2);
\draw (1.5,0.5) -- (1.5,2); 
\draw (2,1) -- (2,2);
\draw (2.5,1.5) -- (2.5,2); 
\draw (3,1.5) -- (3,2);  

\filldraw (1.25,0.75) circle (3pt);
\filldraw (1.75,1.75) circle (3pt);
\filldraw (2.25,1.75) circle (3pt);
\filldraw (0.125,0.625) rectangle (0.375, 0.875);
\filldraw (0.625,-0.375) rectangle (0.875, -0.125);
\node at (0.25,1.75) {$\mbox{X}$};
\node at (0.75,1.75) {$\mbox{X}$};
\node at (1.25,1.75) {$\mbox{X}$};
\node at (0.25,1.25) {$\mbox{X}$};
\node at (0.75,1.25) {$\mbox{X}$};
\node at (1.25,1.25) {$\mbox{X}$};
\node at (-0.25,0.75) {$\mbox{X}$};
\node at (0.75,0.75) {$\mbox{X}$};
\node at (0.75,0.25) {$\mbox{X}$};
\node at (-0.25,-0.25) {$\mbox{X}$};
\node at (0.25,-0.25) {$\mbox{X}$};

\end{tikzpicture}
\caption{The Ferrers board from Figure~\ref{MixedPl} with boxes marked according to class.}\label{MixedPlClass}
\end{figure}
\end{example}

\begin{definition}\label{DefOREMPN}
In the above setting, the {\em $q$-deformed mixed placement numbers} are defined by 
\begin{equation}
m_{k,\ell}(B;q)\equiv m_{k,\ell}^{(\mu,\nu)}(B;q)=\sum_{\phi \in \mathcal{M}_{k,\ell}(B)} w_{\mu,\nu;q}(B,\phi)=\mu^{k}\nu^{\ell}\sum_{\phi \in \mathcal{M}_{k,\ell}(B)} q^{\# \mbox{empty boxes}}.
\end{equation}
\end{definition}
For $q=1$, one obtains $ m_{k,\ell}(B;1)=\mu^{k}\nu^{\ell} |{\mathcal M}_{k,\ell}(B)| $. Furthermore, when only one type of objects (rook or files) is present, these numbers reduce to the ones above,
\begin{equation}\label{ORERed}
m_{k,0}(B;q)=r_k(B;q), \,\, m_{0,\ell}(B;q)=f_{\ell}(B;q).
\end{equation}
For the empty board, one has $m_{0,0}(B;q)=q^{|B|}$ where $|B|$ denotes the number of boxes of the board $B$. In general, when one considers a word $\omega$ having $m$ letters $X$ and $n$ letters $Y$ and a mixed placement $\phi \in  \mathcal{M}_{k,\ell}(B_{\omega}) $, then the resulting normal ordered summand will be, up to some power of $q$, given by $\mu^k \nu^{\ell}Y^{n-k}X^{m-k-\ell}$ (see Example~\ref{ExampleNOOre}). Thus, summing over all possibilities of mixed placements, we have the following result.
\begin{theorem}\label{THNOORE} Let $\omega=Y^{n_r}X^{m_r}\cdots Y^{n_1}X^{m_1}$ (with $n_j,m_j \in \mathbb{N}_0$) be a word in the letters $X$ and $Y$ satisyfing the commutation relation \eqref{CROre} of the $q$-deformed generalized Ore algebra ${\mathcal O}_{\mu,\nu}(q)$. Then the normal ordering coefficients of $\omega$ are given by the $q$-deformed mixed placement numbers,
\begin{equation}\label{NOORE}
\omega = \sum_{k=0}^{\min (|{\bf m}|,|{\bf n}|)}\sum_{\ell=0}^{|{\bf m}|-k} m_{k,\ell}(B_{\omega};q) Y^{|{\bf n}|-k}X^{|{\bf m}|-k-\ell},
\end{equation}
where ${\bf n}=(n_r,\ldots,n_1)$ with $|{\bf n}|=n_1+\cdots+n_r$ (and, similarly, for {\bf m}). 
\end{theorem}
For $q=1$, we can write \eqref{NOORE} also in the following form,
\begin{equation}\label{NOORE1}
\omega = \sum_{k=0}^{\min (|{\bf m}|,|{\bf n}|)}\sum_{\ell=0}^{|{\bf m}|-k} \mu^{k}\nu^{\ell} |{\mathcal M}_{k,\ell}(B_{\omega})| Y^{|{\bf n}|-k}X^{|{\bf m}|-k-\ell}.
\end{equation}
\begin{corollary} For $\nu=0$, the $q$-deformed generalized Ore algebra reduces to the  $q$-deformed Weyl algebra, and \eqref{NOORE} reduces to \eqref{NORQWeyl}. Similarly, for $\mu=0$, the  $q$-deformed generalized Ore algebra reduces to the  $q$-deformed shift algebra, and \eqref{NOORE} reduces to \eqref{NORQShift}.
\end{corollary}
\begin{proof}
Let $\nu=0$. Due to the definition of the $q$-deformed mixed placement numbers only the summand $\ell=0$ remains. According to \eqref{ORERed}, one has $m_{k,0}(B_{\omega};q)=r_{k}(B_{\omega};q)$, hence \eqref{NOORE} reduces to
$$
\omega = \sum_{k=0}^{\min (|{\bf m}|,|{\bf n}|)}  m_{k,0}(B_{\omega};q) Y^{|{\bf n}|-k} X^{|{\bf m}|-k}=\sum_{k=0}^{\min (|{\bf m}|,|{\bf n}|)} r_{k}(B_{\omega};q) Y^{|{\bf n}|-k} X^{|{\bf m}|-k},
$$
which is exactly \eqref{NORQWeyl}. Similarly, when $\mu=0$, only the summand $k=0$ remains. According to \eqref{ORERed}, one has $m_{0,\ell}(B_{\omega};q)=f_{\ell}(B_{\omega};q)$, hence \eqref{NOORE} reduces to
$$
\omega = \sum_{\ell=0}^{|{\bf m}|}  m_{0,\ell}(B_{\omega};q) Y^{|{\bf n}|} X^{|{\bf m}|-\ell}=\sum_{\ell=0}^{|{\bf m}|} f_{\ell }(B_{\omega};q) Y^{|{\bf n}|} X^{|{\bf m}|-\ell},
$$
which is exactly \eqref{NORQShift}.
\end{proof}
Note that we can write \eqref{NOORE} also in the equivalent form
\begin{equation}\label{NOORE2}
\omega = \sum_{r=0}^{|{\bf m}|-m_1}\sum_{t=0}^{\min(|{\bf n}|, r)} m_{t,r-t}(B_{\omega};q) Y^{|{\bf n}|-t}X^{|{\bf m}|-r}.
\end{equation}

Let us consider the word $\omega=(YX)^n$ where the associated Ferrers board is the staircase board $J_n$. We also assume $q=1$. Let $ \phi \in {\mathcal M}_{k,\ell}(J_n)$ be a non-attacking placement of $k$ rooks and $\ell$ files  on $J_n$. Each rook placement diminishes the possible columns and rows for a subsequent placement of files by one unit, see Figure~\ref{StaircaseOre}. Thus, placing for $\phi $ first the $k$ rooks on $J_n$, we have $r_k(J_n)$ possibilities to do this. For each such placement of rooks, we have in the second step $\ell$ files to place on $J_{n-k}$, hence $f_{\ell}(J_{n-k})$ possibilities. Thus, we have the following result.
\begin{proposition} For the set of mixed placements on the staircase board $J_n$ one has
\begin{equation}\label{MixedPlEx}
|{\mathcal M}_{k,\ell}(J_n)| = r_k(J_n) f_{\ell}(J_{n-k}). 
\end{equation}
From \eqref{NOORE1}, one therefore obtains the normal ordering result in  ${\mathcal O}_{\mu,\nu}(1)$,
\begin{equation}\label{StirOre}
(YX)^n =  \sum_{r=0}^{n} \sum_{t=0}^{r} \mu^{n-r} \nu^{r-t} S(n,r) |s(r,t)| Y^{r} X^{t}.
\end{equation}
\end{proposition}
\begin{proof}
It only remains to show \eqref{StirOre}. From \eqref{NOORE1} on obtains, letting $r=n-k$ and then $t=r-\ell$, that 
$$
(YX)^n = \sum_{r=0}^{n} \sum_{t=0}^{r} \mu^{n-r} \nu^{r-t} r_{n-r}(J_n) f_{r-t}(J_{r})  Y^{r}X^{t}.
$$
Using \eqref{StairEq}, this shows the assertion.
\end{proof}

We define, following \cite{TMMS2023}, the {\em Ore-Stirling numbers} $S_{\mu,\nu}(n;j,k) $ as normal ordering coefficients of $(YX)^n$ in ${\mathcal O}_{\mu,\nu}(1)$,
\begin{equation}\label{OreStirlingDef}
(YX)^n = \sum_{j=0}^{n} \sum_{k=0}^{j} S_{\mu,\nu}(n;j,k) Y^{j}X^{k}.
\end{equation}
A comparison with \eqref{StirOre} shows that
\begin{equation}\label{ORESFact}
S_{\mu,\nu}(n;j,k) = \mu^{n-j} \nu^{j-k} S(n,j) |s(j,k)|.
\end{equation}
The factorization of the Ore-Stirling numbers was already derived in \cite[Theorem 3.5]{TMMS2023} in an algebraic fashion (but note that our conventions here are slightly different then the ones in \cite{TMMS2023}, switching, roughly, the role of $X$ and $Y$). We can easily derive from \eqref{ORESFact} a recurrence relation for the Ore-Stirling numbers (note that it differs from the one derived in \cite{TMMS2023} due to the different conventions).
\begin{proposition} The Ore-Stirling numbers satisfy the recurrence relation
\begin{equation}\label{ORESRec}
S_{\mu,\nu}(n+1;j,k) =S_{\mu,\nu}(n;j-1,k-1)+ \mu j S_{\mu,\nu}(n;j,k) + \nu (j-1) S_{\mu,\nu}(n;j-1,k).  
\end{equation}
\end{proposition}
\begin{proof} Use the recurrence relation $S(n+1,j)=S(n,j-1)+jS(n,j)$ to get $S(n+1,j) |s(j,k)|=S(n,j-1)|s(j,k)|+jS(n,j)|s(j,k)|$. Use in the first summand the recurrence relation $|s(j,k)|=|s(j-1,k-1)|+(j-1)|s(j-1,k)|$ to obtain
$$
S(n+1,j) |s(j,k)|=S(n,j-1)|s(j-1,k-1)|+(j-1)S(n,j-1) |s(j-1,k)|+jS(n,j)|s(j,k)|.
$$
Multiplying this with $\mu^{n-j+1} \nu^{j-k}$ and applying \eqref{ORESFact} yields the assertion. 
\end{proof}

We can make the combinatorial interpretation of the Ore-Stirling numbers more explicit by writing $S(n,j)|s(j,k)|=r_{n-j}(J_n)f_{j-k}(J_j)$ and using \eqref{MixedPlEx} to find
\begin{equation}\label{OREEX}
S_{\mu,\nu}(n;j,k) = \mu^{n-j} \nu^{j-k}|{\mathcal M}_{n-j,j-k}(J_n)|,
\end{equation}
i.e., the (weighted) number of non-attacking placements of $n-j$ rooks and $j-k$ files on the staircase board $J_n$. A combinatorial derivation of the recurrence relation of the $q$-deformed Ore-Stirling numbers will be given in Proposition~\ref{PROPQORESRec}.

\begin{figure}[h]
\centering
\begin{tikzpicture}

\draw (0,-1) -- (0,2);
\draw (0.5,-0.5) -- (0.5,2);
\draw (1,0) -- (1,2);   
\draw (1.5,0.5) -- (1.5,2);  
\draw (2,1) -- (2,2); 
\draw (2.5,1.5) -- (2.5,2); 
\draw (0,2) -- (3,2);
\draw (0,1.5) -- (2.5,1.5);
\draw (0,1) -- (2,1);
\draw (0,0.5) -- (1.5,0.5);
\draw (0,0) -- (1,0);
\draw (0,-0.5) -- (0.5,-0.5);

\filldraw (0.625,0.125) rectangle (0.875, 0.375);
\filldraw (1.125,1.125) rectangle (1.375, 1.375);

\node at (0.25, 1.75) {$\checkmark $};
\node at (0.25, 0.75) {$\checkmark $};
\node at (0.25,-0.25) {$\checkmark $};
\node at (1.75, 1.75) {$\checkmark $};
\node at (1.75, 1.25) {$\checkmark $};
\node at (2.25, 1.75) {$\checkmark $};

\node at (3.75,0.75) {$\rightsquigarrow$};

\draw (4.5,0.5) -- (4.5,2);
\draw (4.5,0) -- (4.5,0.5);
\draw (5,1) -- (5,2);
\draw (5,0.5) -- (5,1);
\draw (5.5,1.5) -- (5.5,2);   
\draw (5.5,1) -- (5.5,1.5);
\draw (6,1.5) -- (6,2);  
\draw (4.5,2) -- (6,2);
\draw (6,2) -- (6.5,2);
\draw (4.5,1.5) -- (5.5,1.5);
\draw (5.5,1.5) -- (6,1.5); 
\draw (4.5,1) -- (5,1);
\draw (5,1) -- (5.5,1); 
\draw (4.5,0.5) -- (5,0.5);

\end{tikzpicture}
\caption{After placing two rooks on $J_5$, the set of allowed file placements (marked with $\checkmark$ on the left) is equal to the set of file placements on $J_3$.}\label{StaircaseOre}
\end{figure}
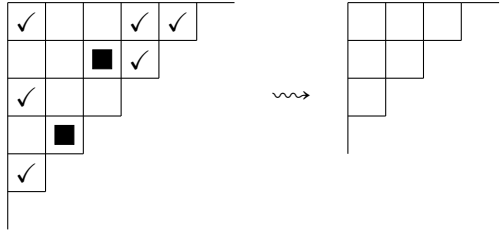
 
 \begin{remark}
 In \cite{TMMS2023}, a combinatorial interpretation for the Ore-Stirling numbers was discussed, but an intepretation for the normal ordering coefficients of arbitrary words in $X$ and $Y$ was lacking and proposed as an open problem (see Point 5 in  \cite[Section 5]{TMMS2023}). This sought-for interpretation is given in Theorem~\ref{THNOORE} (for the $q$-deformed case).
 \end{remark}
 
We can define in analogy to \eqref{OreStirlingDef} the {\em $q$-deformed Ore-Stirling numbers} $S_{\mu,\nu;q}(n;j,k) $ as normal ordering coefficients of $(YX)^n$ in ${\mathcal O}_{\mu,\nu}(q)$, 
\begin{equation}\label{QOreStirlingDef}
(YX)^n = \sum_{j=0}^{n} \sum_{k=0}^{j} S_{\mu,\nu;q}(n;j,k) Y^{j}X^{k}.
\end{equation}
\begin{proposition}\label{PropOS} The $q$-deformed Ore-Stirling numbers are given by
\begin{equation}\label{QOREEX}
S_{\mu,\nu;q}(n;j,k)=m_{n-j,j-k}(J_n;q). 
\end{equation}
\end{proposition}
\begin{proof}
Switching the variables as $r=n-k$ and $t=r-\ell$, we obtain from \eqref{NOORE} that
$$
(YX)^n = \sum_{r=0}^{n}\sum_{t=0}^{r} m_{n-r,r-t}(J_n;q) Y^{r}X^{t}.
$$
A comparison with the definition of $S_{\mu,\nu;q}(n;j,k)$ in \eqref{QOreStirlingDef} yields the assertion.
\end{proof}
For $q=1$, \eqref{QOREEX} reduces to  \eqref{OREEX} since $m_{n-j,j-k}(J_n;1)=\mu^{n-j}\nu^{j-k}|{\mathcal M}_{n-j,j-k}(J_n)|$. Let us point out that Briand \cite[Theorem 3.1]{Briand2025} showed that normal ordering $(UD)^n$ when $DU=qUD+I+D$ can be described in terms of placements of rooks and ``lances'' on $J_n$.

\begin{example}
Let $n=3$. Using the commutation relation \eqref{CROre}, one can show that in ${\mathcal O}_{\mu,\nu}(q)$
$$
(YX)^3=q^3Y^3X^3+(q+2q^2)\nu Y^3X^2+(1+q)\nu^2 Y^3X+(2q+q^2)\mu Y^2X^2+(2+q)\mu \nu Y^2X+\mu^2 YX.
$$ 
There are 13 summands. Each summand corresponds to a non-attacking placement of rooks and files on $J_3$.  The empty placement corresponding to $m_{0,0}(J_3;q)$ has weight $q^3$. The remaining 12 placements are shown in Figure~\ref{MixedFig} with their weights. For example, the coefficient of $Y^2X$ is given by $S_{\mu,\nu;q}(3;2,1)=m_{1,1}(J_3;q)=\mu\nu+q\mu\nu+\mu\nu=(2+q)\mu\nu$ (the last 3 diagrams of the second row).

\begin{figure}[h]
\centering
\begin{tikzpicture}

\draw (0,0.5) -- (0,2);
\draw (0.5,1) -- (0.5,2);
\draw (1,1.5) -- (1,2);
\draw (0,2) -- (1.5,2);
\draw (0,1.5) -- (1,1.5);
\draw (0,1) -- (0.5,1);
\filldraw (0.625,1.625) rectangle (0.875, 1.875);
\node at (0.25,01.75) {$\mbox{X}$};
\node at (0.75,0.25) {$q\mu$};

\draw (2.5,0.5) -- (2.5,2);
\draw (3,1) -- (3,2);
\draw (3.5,1.5) -- (3.5,2);
\draw (2.5,2) -- (4,2);
\draw (2.5,1.5) -- (3.5,1.5);
\draw (2.5,1) -- (3,1);
\filldraw (2.625,1.125) rectangle (2.875, 1.375);
\node at (2.75,01.75) {$\mbox{X}$};
\node at (3.25,0.25) {$q\mu$};

\draw (5,0.5) -- (5,2);
\draw (5.5,1) -- (5.5,2);
\draw (6,1.5) -- (6,2);
\draw (5,2) -- (6.5,2);
\draw (5,1.5) -- (6,1.5);
\draw (5,1) -- (5.5,1);
\filldraw (5.25,1.75) circle (1.5pt);
\filldraw (5.125,1.625) rectangle (5.375, 1.875);
\node at (5.75,0.25) {$ q^2 \mu$};

\draw (7.5,0.5) -- (7.5,2);
\draw (8,1) -- (8,2);
\draw (8.5,1.5) -- (8.5,2);
\draw (7.5,2) -- (9,2);
\draw (7.5,1.5) -- (8.5,1.5);
\draw (7.5,1) -- (8,1);
\filldraw (8.25,1.75) circle (3pt);
\node at (8.25,0.25) {$ q^2\nu$};

\draw (10,0.5) -- (10,2);
\draw (10.5,1) -- (10.5,2);
\draw (11,1.5) -- (11,2);
\draw (10,2) -- (11.5,2);
\draw (10,1.5) -- (11,1.5);
\draw (10,1) -- (10.5,1);
\filldraw (10.25,1.25) circle (3pt);
\node at (10.25,01.75) {$\mbox{X}$};
\node at (10.75,0.25) {$ q\nu$};

\draw (12.5,0.5) -- (12.5,2);
\draw (13,1) -- (13,2);
\draw (13.5,1.5) -- (13.5,2);
\draw (12.5,2) -- (14,2);
\draw (12.5,1.5) -- (13.5,1.5);
\draw (12.5,1) -- (13,1);
\filldraw (12.75,1.75) circle (3pt);
\node at (13.25,0.25) {$ q^2\nu $};

Second row
\draw (0,-2) -- (0,-0.5);
\draw (0.5,-1.5) -- (0.5,-0.5);
\draw (1,-1) -- (1,-0.5);
\draw (0,-0.5) -- (1.5,-0.5);
\draw (0,-1) -- (1,-1);
\draw (0,-1.5) -- (0.5,-1.5);
\filldraw (0.75,-0.75) circle (1.5pt);
\filldraw (0.25,-1.25) circle (1.5pt);
\filldraw (0.625,-0.875) rectangle (0.875, -0.625);
\filldraw (0.125,-1.375) rectangle (0.375, -1.125);
\node at (0.25,-0.75) {$\mbox{X}$};
\node at (0.75,-2.25) {$\mu^2$};

\draw (2.5,-2) -- (2.5,-0.5);
\draw (3,-1.5) -- (3,-0.5);
\draw (3.5,-1) -- (3.5,-0.5);
\draw (2.5,-0.5) -- (4,-0.5);
\draw (2.5,-1) -- (3.5,-1);
\draw (2.5,-1.5) -- (3,-1.5);
\filldraw (2.75,-1.25) circle (3pt);
\filldraw (3.25,-0.75) circle (3pt);
\node at (2.75,-0.75) {$\mbox{X}$};
\node at (3.25,-2.25) {$\nu^2$};

\draw (5,-2) -- (5,-0.5);
\draw (5.5,-1.5) -- (5.5,-0.5);
\draw (6,-1) -- (6,-0.5);
\draw (5,-0.5) -- (6.5,-0.5);
\draw (5,-1) -- (6,-1);
\draw (5,-1.5) -- (5.5,-1.5);
\filldraw (5.25,-0.75) circle (3pt);
\filldraw (5.75,-0.75) circle (3pt);
\node at (5.75,-2.25) {$ q \nu^2 $};

\draw (7.5,-2) -- (7.5,-0.5);
\draw (8,-1.5) -- (8,-0.5);
\draw (8.5,-1) -- (8.5,-0.5);
\draw (7.5,-0.5) -- (9,-0.5);
\draw (7.5,-1) -- (8.5,-1);
\draw (7.5,-1.5) -- (8,-1.5);
\filldraw (8.25,-0.75) circle (3pt);
\filldraw (7.625,-1.375) rectangle (7.875, -1.125);
\node at (7.75,-0.75) {$\mbox{X}$};
\node at (8.25,-2.25) {$\mu \nu$};

\draw (10,-2) -- (10,-0.5);
\draw (10.5,-1.5) -- (10.5,-0.5);
\draw (11,-1) -- (11,-0.5);
\draw (10,-0.5) -- (11.5,-0.5);
\draw (10,-1) -- (11,-1);
\draw (10,-1.5) -- (10.5,-1.5);
\filldraw (10.75,-0.75) circle (3pt);
\filldraw (10.125,-0.875) rectangle (10.375, -0.625);
\node at (10.75,-2.25) {$q\mu \nu$};

\draw (12.5,-2) -- (12.5,-0.5);
\draw (13,-1.5) -- (13,-0.5);
\draw (13.5,-1) -- (13.5,-0.5);
\draw (12.5,-0.5) -- (14,-0.5);
\draw (12.5,-1) -- (13.5,-1);
\draw (12.5,-1.5) -- (13,-1.5);
\filldraw (12.75,-1.25) circle (3pt);
\filldraw (13.125,-0.875) rectangle (13.375, -0.625);
\node at (12.75,-0.75) {$\mbox{X}$};
\node at (13.25,-2.25) {$\mu \nu$};

\end{tikzpicture}
\caption{All nontrivial non-attacking mixed placements of rooks and files on $J_3$.}\label{MixedFig}
\end{figure}
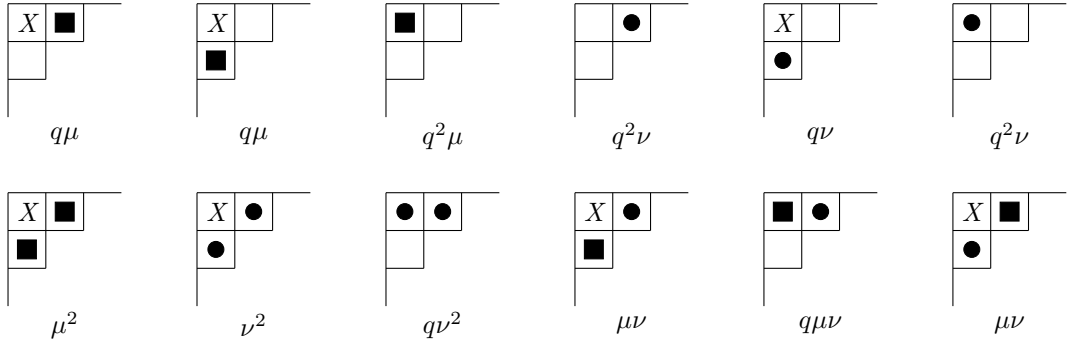
\end{example}
A few values of $S_{\mu,\nu;q}(n;j,k)$ are easy to determine:
\begin{itemize}
\item $S_{\mu,\nu;q}(n;n,n)=m_{0,0}(J_n;q)=q^{|J_n|}$, giving the summand $q^{\binom{n}{2}} Y^nX^n$,
\item $S_{\mu,\nu;q}(n;1,1)=m_{n-1,0}(J_n;q)=\mu^{n-1}$, giving the summand $\mu^{n-1} YX$,
\item $S_{\mu,\nu;q}(n;n,n-\ell)=m_{0,\ell}(J_n;q)=f_{\ell}(J_n;q)$, giving the summand $f_{\ell}(J_n;q)Y^nX^{n-\ell}$, 
\item $S_{\mu,\nu;q}(n;n-k,n-k)=m_{k,0}(J_n;q)=r_{k}(J_n;q)$, giving the summand $r_{k}(J_n;q)Y^{n-k}X^{n-k}$.
\end{itemize}
Recall that the {\em $q$-deformed numbers} are defined by $[n]_q=1+q+\cdots + q^{n-1}=\frac{1-q^n}{1-q}$. Furthermore, 
$$
[n]_q!=[n]_q [n-1]_q\cdots [2]_q[1]_q, \, \,\,\,\, \binom{m}{k}_q =\frac{[n]_q!}{[n-k]_q![k]_q!}.
$$ 
For the $q$-deformed Ore-Stirling numbers one has the following analogue of \eqref{ORESRec} (and to which it reduces for $q=1$).
\begin{proposition}\label{PROPQORESRec} The $q$-deformed Ore-Stirling numbers satisfy the recurrence relation
\begin{equation}\label{QORESRec}
S_{\mu,\nu;q}(n+1;j,k) =q^{j-1}S_{\mu,\nu;q}(n;j-1,k-1)+\mu [j]_q S_{\mu,\nu;q}(n;j,k) +\nu [j-1]_q S_{\mu,\nu;q}(n;j-1,k).  
\end{equation}
\end{proposition}
\begin{proof} We give a combinatorial proof using \eqref{QOREEX} and the definition of the $m_{r,s}(B;q)$,
$$
S_{\mu,\nu;q}(n+1;j,k)=m_{n+1-j,j-k}(J_{n+1};q)=\mu^{n+1-j}\nu^{j-k}\sum_{\phi \in \mathcal{M}_{n+1-j,j-k}(J_{n+1})} q^{\# \mbox{empty boxes}}.
$$ 
Thus, we have to consider all mixed placements $\phi \in \mathcal{M}_{n+1-j,j-k}(J_{n+1})$. We interprete the staircase board $J_{n+1}$ as staircase board $J_n$ where we have adjoined a column with $n$ boxes to the left. Let us denote this column for brevity by $C_n$ such that we write in an informal fashion $J_{n+1}=C_n \oplus J_n$. Then there are 3 types of placements of $(n+1-j)$ rooks and $(j-k)$ files on $J_{n+1}$:
\begin{itemize}
\item Type $I$: Place all rooks and files on $J_n$ and leave $C_n$ empty.
\item Type $II$: Place $(n-j)$ rooks and $(j-k)$ files on $J_{n}$ and one rook in $C_n$.
\item Type $III$: Place $(n+1-j)$ rooks and $(j-k-1)$ files on $J_{n}$ and one file in $C_n$.
\end{itemize} 
Thus, we can write
$$
\mathcal{M}_{n+1-j,j-k}(J_{n+1})=\mathcal{M}_{n+1-j,j-k}^{I}(J_{n+1})\cup\mathcal{M}_{n+1-j,j-k}^{II}(J_{n+1})\cup  \mathcal{M}_{n+1-j,j-k}^{III}(J_{n+1}).
$$
Let us start with placements of type $I$. Each such placement  $\phi  \in \mathcal{M}_{n+1-j,j-k}^{I}(J_{n+1})$ corresponds to a placement $\phi' \in \mathcal{M}_{n-(j-1),(j-1)-(k-1)}(J_{n})$. If $\phi'$ has weight $\omega_{\mu,\nu;q}(J_n,\phi')$ then $\omega_{\mu,\nu;q}(J_{n+1},\phi)=q^{j-1}\omega_{\mu,\nu;q}(J_n,\phi')$ since in the column $C_n$ there are additional $n-(n+1-j)=j-1$ empty boxes. Thus, the contribution of all placements of type $I$ is given by
$$
\mu^{n+1-j}\nu^{j-k}q^{j-1}\sum_{\phi' \in \mathcal{M}_{n-(j-1),(j-1)-(k-1)}(J_{n})} q^{\# \mbox{empty boxes}} =q^{j-1}S_{\mu,\nu;q}(n;j-1,k-1).
$$
Let us turn to placements of type $II$. Starting from a placement $\phi' $ of $(n-j)$ rooks and $(j-k)$ files on $J_{n}$, i.e., 
$\phi' \in \mathcal{M}_{n-j,j-k}(J_{n})$ we obtain $j$ placements $\phi$ of tpye $II$ on $J_{n+1}$ since there are $n-(n-j)=j$ possible boxes to place the rook in $C_n$. Placing the rook in the $r$-th possible row from above (where $r=1,\ldots,j$) will result in $j-r$ additional empty boxes. Thus, the sum of the weights of these $j$ placements $\phi$ is given by $\mu (1+q+\cdots+q^{j-1})\omega_{\mu,\nu;q}(J_n,\phi')=\mu [j]_q \omega_{\mu,\nu;q}(J_n,\phi')$. The contribution of all placements of type $II$ is, therefore, given by
$$
\mu^{n+1-j}\nu^{j-k}\mu [j]_q \sum_{\phi' \in \mathcal{M}_{n-j,j-k}(J_{n})} q^{\# \mbox{empty boxes}} = \mu [j]_q S_{\mu,\nu;q}(n;j,k).
$$
For placements of type $III$ the argument is the same as for those of type $II$. Here we have only $(j-1)$ possible boxes in $C_n$ to place the file. Thus, for each placement $\phi' \in \mathcal{M}_{n-(j-1),(j-1)-k}(J_{n})$ we obtain $(j-1)$ placements $\phi$ of tpye $III$, and the sum of the weights of these $(j-1)$ placements $\phi$ is given by $\nu [j-1]_q \omega_{\mu,\nu;q}(J_n,\phi')$. Thus, the contribution of all placements of type $III$ is given by
$$
\mu^{n+1-j}\nu^{j-k} \nu [j-1]_q \sum_{\phi' \in \mathcal{M}_{n-(j-1),(j-1)-k}(J_{n})} q^{\# \mbox{empty boxes}} = \nu [j-1]_q S_{\mu,\nu;q}(n;j-1,k).
$$
 Summing the contributions of placements of type $I,II$ and $III$ gives the assertion.
\end{proof}

In the Weyl algebra ${\mathcal W}={\mathcal O}_{1,0}(1)$, generalized Stirling numbers $S_{r,s}(n,k)$ have been introduced as normal ordering coefficients of $(Y^rX^s)^n$ (see \cite{TMMS2016} for the remarkable history of these numbers starting with Carlitz \cite{LC1930,LC1932} and McCoy \cite{Coy1934} in the early 1930s and \cite{BLFL2011} for many combinatorial interpretations). In the shift algebra ${\mathcal S}={\mathcal O}_{0,1}(1)$, analogous numbers $\sigma_{r,s}(n,k)$ were introduced as normal ordering coefficients of $(Y^rX^s)^n$ \cite{MS2022}. Thus, it is natural to define {\em $q$-deformed generalized Ore-Stirling numbers} $S_{\mu,\nu;q}^{r,s}(n;j,k)$ as normal ordering coefficients of  $(Y^rX^s)^n$ in ${\mathcal O}_{\mu,\nu}(q)$. We specialize to $r=2,s=1$, since it is well known that $S_{2,1}(n,k)=L(n,k)$, the (unsigned) Lah numbers (see \cite{TMMS2016} and the references given therein). Thus, one has 
\begin{equation}\label{Lahrook}
(Y^2X)^n=\sum_{k=1}^n L(n,k) Y^{n+k}X^k=\sum_{k=1}^n r_{n-k}({\mathcal L}_n) Y^{n+k}X^k
\end{equation}
where we used in the second equation \eqref{NORWeyl} and ${\mathcal L}_n$ denotes the {\em Lah board} associated to the word $(Y^2X)^n$ (see Figure~\ref{Laguerre} for ${\mathcal L}_3$ where the outlining path is drawn fat). In general, the {\em $m$-jump board} $J_{n,m}$ has column heights $(m(n-1),m(n-2),\ldots,m,0)$ \cite{JGJH2000} and, for $m=1$, one obtains the staircase board, $J_{n,1}=J_n$. For $m=2$, we have $J_{n,2}={\mathcal L}_n$ and we coined it for obvious reasons Lah board. Thus, from \eqref{Lahrook} we obtain the well known relation $L(n,k)= r_{n-k}({\mathcal L}_n)$. 

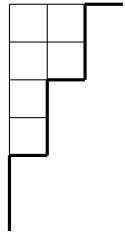
\begin{figure}[h]
\centering
\begin{tikzpicture}
\draw (0,1) -- (0,3);
\draw[very thick] (0,0) -- (0,1);
\draw (0.5,2) -- (0.5,3);
\draw[very thick] (0.5,1) -- (0.5,2);
\draw[very thick] (1,2) -- (1,3);

\draw (0,3) -- (1,3);
\draw[very thick] (1,3) -- (1.5,3);
\draw (0,2.5) -- (1,2.5);
\draw (0,2) -- (0.5,2);
\draw[very thick] (0.5,2) -- (1,2); 
\draw (0,1.5) -- (0.5,1.5);
\draw[very thick] (0,1) -- (0.5,1);
\end{tikzpicture}
\caption{The Lah board ${\mathcal L}_3$ associated to $(Y^2X)^3$.}\label{Laguerre}
\end{figure}
We define the {\em $q$-deformed Ore-Lah numbers} $L_{\mu,\nu;q}(n;j,k)$ by analogy to \eqref{QOreStirlingDef},
\begin{equation}\label{QOreLahDef}
(Y^2X)^n = \sum_{j=0}^{2n} \sum_{k=0}^{j} L_{\mu,\nu;q}(n;j,k) Y^{j}X^{k}.
\end{equation}

In analogy to Proposition~\ref{PropOS}, we have the following result.
\begin{proposition} The $q$-deformed Ore-Lah numbers are given by
\begin{equation}\label{QORELAHEX}
L_{\mu,\nu;q}(n;j,k)=m_{2n-j,j-n-k}({\mathcal L}_n;q). 
\end{equation}
\end{proposition}
For $\nu =0$, only the term with $j-n-k=0$ remains, thus only $L_{\mu,0;q}(n;n+k,k)=m_{n-k,0}({\mathcal L}_n;q)=r_{n-k}({\mathcal L}_n;q)$, see \eqref{ORERed}. Furthermore, letting $q=1$ and $\mu=1$, one finds $L_{1,0;1}(n;n+k,k)=r_{n-k}({\mathcal L}_n)=L(n,k)$, as expected. On the other hand, for $\mu=0$, only the term with $j=2n$ remains, thus only $L_{0,\nu;q}(n;2n,k)=m_{0,n-k}({\mathcal L}_n;q)=f_{n-k}({\mathcal L}_n;q)$, see \eqref{ORERed}.

\begin{proposition}\label{OreLahRec} The $q$-deformed Ore-Lah numbers satisfy the recurrence relation
$$
L_{\mu,\nu;q}(n+1;j,k) =q^{j-2}L_{\mu,\nu;q}(n;j-2,k-1)+\mu [j-1]_q L_{\mu,\nu;q}(n;j-1,k) +\nu [j-2]_q L_{\mu,\nu;q}(n;j-2,k).  
$$
\end{proposition}
\begin{proof}
The proof is very similar to the one for Proposition~\ref{PROPQORESRec}. Note that ${\mathcal L}_{n+1}=C_{2n}\oplus {\mathcal L}_n$ (see Figure~\ref{Laguerre} for $n=2$). From \eqref{QORELAHEX}, we have to consider $L_{\mu,\nu;q}(n+1;j,k)=m_{2(n+1)-j,j-(n+1)-k}({\mathcal L}_{n+1};q)$. For a placement of Type $I$, neither a rook nor a file is placed in $C_{2n}$ and there are $2n-(2n+2-j)=j-2$ empty boxes in $C_{2n}$. Thus, the contribution of all placements of type $I$ is
$$
q^{j-2}m_{2n-(j-2),(j-2)-n-(k-1)}({\mathcal L}_{n};q)=q^{j-2}L_{\mu,\nu;q}(n;j-2,k-1).
$$
For a placement of Type $II$, a rook is placed in $C_{2n}$ and the remaining rooks and files are placed in ${\mathcal L}_{n}$.  In $C_{2n}$, there are $j-1$ boxes to place the rook. The contribution of all placements of type $II$ is
$$
\mu [j-1]_q m_{2n-(j-1),(j-1)-n-k}({\mathcal L}_{n};q)=\mu [j-1]_q L_{\mu,\nu;q}(n;j-1,k).
$$ 
For a placement of Type $III$, a file is placed in $C_{2n}$ and the remaining rooks and files are placed in ${\mathcal L}_{n}$.  In $C_{2n}$, there are $j-2$ boxes to place the file. The contribution of all placements of type $III$ is
$$
\nu [j-2]_q m_{2n-(j-2),(j-2)-n-k}({\mathcal L}_{n};q)=\nu [j-2]_q L_{\mu,\nu;q}(n;j-2,k).
$$ 
Adding the contributions of type $I,II$ and $III$ yields the assertion.
\end{proof}
We specialize the above recurrence to $\nu=0$ where only $j=n+k$ remains and let $\mu=1$. Define $L_q(n,k)\equiv L_{1,0;q}(n;n+k,k)$. Since $L_q(n+1,k)=L_{1,0;q}(n+1;n+1+k,k)$ the above recurrence relation yields
\begin{align*}
L_q(n+1,k)=& q^{n+k-1}L_{1,0;q}(n;n+k-1,k-1)+[n+k]_q L_{1,0;q}(n;n+k,k) \\
 =& q^{n+k-1}L_q(n,k-1)+[n+k]_q L_q(n,k),
\end{align*}
which is the usual recurrence relation of the $q$-deformed Lah numbers (see, e.g., \cite{TMMS2016}). We can also specialize to $\mu=0$ where only $j=2n$ remains and let $\nu=1$. Define $L^{s}_q(n,k)\equiv L_{0,1;q}(n;2n,k)$. Then
$$
L^{s}_q(n+1,k)=q^{2n}L^{s}_q(n,k-1)+[2n]_q L^{s}_q(n,k).
$$

Let us briefly consider $(Y^rX)^n$, for arbitrary $r\in \mathbb{N}$. As mentioned above, the normal ordering coefficients of $(Y^rX)^n$ in the Weyl algebra ${\mathcal W}={\mathcal O}_{1,0}(1)$ are the generalized Stirling numbers $S_{r,1}(n,k)$, and it would be appropriate to call them {\em Scherk numbers of order $r$}, $S^{(r)}(n,k)$. Clearly, $(Y^rX)^n$ outlines the $r$-jump board $J_{n,r}$, and one has $S^{(r)}(n,k)=r_{n-k}(J_{n,r})$. Defining in complete analogy to \eqref{QOreLahDef} the {\em $q$-deformed Ore-Scherk numbers of degree $r$}, $S^{(r)}_{\mu,\nu;q}(n;j,k)$, as normal ordering coefficients of $(Y^rX)^n$ in ${\mathcal O}_{\mu,\nu}(q)$, one can show in the same fashion as above the following result.
\begin{proposition} The $q$-deformed Ore-Scherk numbers of degree $r$ are given by
\begin{equation}\label{OREScherk}
S^{(r)}_{\mu,\nu;q}(n;j,k)=m_{rn-j,j-(r-1)n-k}(J_{n,r};q).
\end{equation}
\end{proposition}
For $r=1$, one recovers \eqref{QOREEX}, and for $r=2$, one recovers \eqref{QORELAHEX}. Let $q=\mu=1$ and let us specialize to $\nu=0$ so that only $j=(r-1)n+k$ remains. Thus, $S^{(r)}_{1,0;1}(n;(r-1)n+k,k)=m_{n-k,0}(J_{n,r};1)=r_{n-k}(J_{n,r})=S^{(r)}(n,k)$.

Now, let us turn to the basic word $X^mY^n$, for $m,n\in \mathbb{N}$. The associated Ferrers board is the {\em rectangle board} $R_{m,n}$ having $m$ columns of height $n$. From \eqref{NOORE2}, we obtain for $q=1$ that in ${\mathcal O}_{\mu,\nu}(1)$
\begin{equation}\label{BasicCR}
X^mY^n = \sum_{r=0}^{m}\sum_{t=0}^{\min(n, r)} \mu^{t} \nu^{r-t} | {\mathcal M}_{t,r-t}(R_{m,n})| Y^{n-t}X^{m-r},
\end{equation}
and it remains to determine $| {\mathcal M}_{t,r-t}(R_{m,n})|$, the number of mixed placements of $t$ rooks and $(r-t)$ files on $R_{m,n}$. The rook and file numbers of $R_{m,n}$ are given by
\begin{equation}\label{Rectangle}
r_k(R_{m,n})=k! \binom{n}{k}\binom{m}{k}, \,\,\,\,\,  f_k(R_{m,n})= n^k \binom{m}{k}. 
\end{equation}
Let us recall the simple argument: To place $k$ files on $R_{m,n}$ we first choose $k$ of the $m$ columns, giving the factor $\binom{m}{k}$. In each of the $k$ columns we have $n$ possibilities to place the file, hence the factor $n^k$, yielding in total $ n^k \binom{m}{k}$ possibilities. For the placement of $k$ non-attacking rooks, after choosing $k$ of the $m$ columns,  we have to respect that no two rooks lie in the same row. After placing the first rook, we have $(n-1)$ possible rows to place the second rook, $(n-2)$ possible rows to place the third rook, hence, $n(n-1)(n-2)\cdots (n-k+1)=k!\binom{n}{k}$ possibilities to place the $k$ rooks in the chosen columns, giving in total $k! \binom{n}{k}\binom{m}{k}$ possibilities.  

To determine the number of mixed placements of $k$ rooks and $\ell$ files is more complicated since a file can be placed in the row of a rook to the right of the rook -- but not to the left of the rook. Thus, after having placed $k$ rooks, when placing a file to the right of all rooks, we have $n$ possibilities to choose from, when placing it between the $j$-th and $(j+1)$-th rook from the right (with $1\leq j \leq k$) we have in the column $(n-j)$ possibilities to place it, and when placing it to the left of all rooks, there are $(n-k)$ possibilities to place it. Let us introduce some terminology. We denote by ${\mathcal C}_k(m)$ the set of all subsets of $\{1,2,\ldots,m\}$ with $k$ elements. We number the columns of $R_{m,n}$ from right to left, starting with 1 for the right-most column. Let us place the $k$ rooks first. The choice of $k$ columns corresponds to $\pi = (\pi_1,\ldots,\pi_k) \in {\mathcal C}_k(m)$. To distribute $\ell$ files in the remaining $(m-k)$ colums, we consider the (at most) $(k+1)$ {\em blocks} induced by the placement of the $k$ rooks: The first block contains $m_1(\pi)=(\pi_1-1)$ columns to the right of all rooks, the $j$-th block, for $j=2,\ldots,k$, contains $m_j(\pi)=(\pi_{j}-\pi_{j-1}-1)$ columns between the $(j-1)$-th and $j$-th rook, and the $(k+1)$-th block contains $m_{k+1}(\pi)=(m-\pi_k)$ columns to the left of all rooks. Note that each of these blocks can be empty (when two rooks lie in neighbouring columms). By construction, $m_1(\pi)+\cdots + m_{k+1}(\pi)=m-k$. We distribute the $\ell$ files into the $(k+1)$ blocks. This corresponds to a weak composition of $\ell$ into $(k+1)$ parts, where the $j$-th part corresponds to the number of files in the $j$-th block. Let us denote the set of weak compositions of $\ell$ into $(k+1)$ parts by ${\mathcal WC}_{k+1}(\ell)=\{\rho =(\rho_1,\ldots, \rho_{k+1}) \in \mathbb{N}_0^{k+1} \, | \, \rho_1+\cdots + \rho_{k+1} = \ell\} $. However, for a given $\pi \in {\mathcal C}_k(m)$ not all weak compositions can appear but only those suiting $\pi$. Thus, we define 
$$
{\mathcal WC}_{k+1}(\ell || \pi ) =\{\rho \in {\mathcal WC}_{k+1}(\ell) \, | \, \rho_j \leq m_j(\pi), j=1,\ldots, k+1 \}.
$$
\begin{proposition} Let $k,\ell \in \mathbb{N}_0$. The number of mixed placements of $k$ rooks and $\ell$ files on the rectangle board $R_{m,n}$ is given by
\begin{equation}\label{MRRectangle}
|{\mathcal M}_{k,\ell}(R_{m,n})|= \sum_{\pi \in {\mathcal C}_k(m)} \sum_{\rho \in {\mathcal WC}_{k+1}(\ell || \pi)}k!  \binom{n}{k} \prod_{j=1}^{k+1} \binom{m_{j}(\pi)}{\rho_{j}}(n+1-j)^{\rho_{j}}.
\end{equation}
\end{proposition}
\begin{proof}
We continue the argument from above and note that we have $ k! \binom{n}{k} | {\mathcal C}_k(m) |$ possibilities to place the $k$ rooks. However, for each $\pi \in {\mathcal C}_k(m) $ we have to determine the possible placements of $\ell$ files in the remaining $(m-k)$ columns. Each weak composition $\rho =(\rho_1,\ldots, \rho_{k+1}) \in {\mathcal WC}_{k+1}(\ell || \pi )$ suiting $\pi$ corresponds to a distribution of $\rho_j$ files into the $j$-th block, for $j=1,\ldots,k+1$. Since the $j$-th block has $m_j(\pi)$ columns, there are $\binom{m_j(\pi)}{\rho_j}$ choices of columns to place the files. Each column in the $j$-th block has $(n+1-j)$ boxes where the file can be placed, giving a further factor of $(n+1-j)^{\rho_j}$, yielding in total $\binom{m_j(\pi)}{\rho_j} (n+1-j)^{\rho_j}$ for the $j$-th block. Considering all blocks and summing over all weak compositions suiting $\pi$ yields the assertion.    
\end{proof}
\begin{corollary}
For $\ell=0$, only $\rho =(0,\ldots,0)$ remains, implying $|{\mathcal M}_{k,0}(R_{m,n})|= k!  \binom{n}{k} |{\mathcal C}_k(m) | =k!  \binom{n}{k} \binom{m}{k}$, recovering the first equation of \eqref{Rectangle}. For $k=0$, only $\rho =(\ell)$ remains, thus $|{\mathcal M}_{0,\ell}(R_{m,n})|= \binom{m}{\ell}n^{\ell}$, recovering  the second equation of \eqref{Rectangle}.
\end{corollary}
Inserting \eqref{MRRectangle} into \eqref{BasicCR} yields in ${\mathcal O}_{\mu,\nu}(1)$ the basic commutation relation
$$
X^mY^n = \sum_{r=0}^{m}\sum_{t=0}^{\min(n, r)}\left\{ \sum_{\pi \in {\mathcal C}_t(m)} \sum_{\rho \in {\mathcal WC}_{t+1}(r-t || \pi)}  \mu^{t} \nu^{r-t} t!  \binom{n}{t} \prod_{j=1}^{t+1} \binom{m_{j}(\pi)}{\rho_{j}}(n+1-j)^{\rho_{j}} \right\} Y^{n-t}X^{m-r}.
$$
\begin{example}
Consider $m=n=1$. On the right-hand side, only three combinations contribute: 1) $r=t=0$ yields the summand $YX$, 2) $r=1,t=0$ the summand $\nu Y$, and 3) $r=t=1$ the summand $\mu I $. Thus, $XY=YX+\nu Y+\mu I$, as it should.
\end{example}
\begin{remark}
Special rectangular boards have been considered in literature. For example, the {\em Abel board} $A_n \equiv R_{n-1,n}$ has been studied in \cite{JGJH2000,MSMY2017}. From \eqref{Rectangle}, one has $f_{n-k}(A_n)=n^{n-k} \binom{n-1}{k-1}=t_{n,k}$, the number of labeled forests on $n$ vertices composed of $n-k$ rooted trees. More generally, {\em $r$-restricted Abel boards} $A_n^{(r)}\equiv R_{n-r,n}$ were studied in \cite{MSMY2017}. One has $f_{n-k}(A_n^{(r)})=n^{n-k} \binom{n-r}{k-r}=t_{n,k}^{(r)}$. Schlosser and Yoo \cite{MSMY2017} suggested to refer to $t_{n,k}$ as {\em Abel numbers} and to $t_{n,k}^{(r)}$ as {\em $r$-restricted Abel numbers}. Closely related is the {\em Laguerre board} $L_n \equiv R_{n,n-1}$ \cite{GR1986}. From \eqref{Rectangle}, one has $r_{n-k}(L_n)=\frac{n!}{k!}\binom{n-1}{k-1}=L(n,k)$, the (unsigned) Lah numbers. The {\em $r$-restricted Laguerre boards} $L_n^{(r)}\equiv R_{n+r-1,n-r}$ were introduced in \cite{MSMY2017}, and $r_{n-k}(L_n^{(r)})=\frac{(n+r-1)!}{(k+r-1)!}\binom{n-r}{k-r}=L_r(n,k)$, recovering the $r$-Lah numbers.  
\end{remark}

As mentioned in the introduction, Yeliussizov \cite[Lemma 10.2]{DY2019} gave a formula for $D^nU^m$ where $DU-UD=I+D$ involving the Eulerian numbers. We adapt the proof given in  \cite{DY2019} to the present case and obtain the following result.
\begin{proposition}
Let $X$ and $Y$ satisfy the commutation relation $XY-YX=I+Y$ of the Ore algebra ${\mathcal O}_{1,1}(1)$. Then one has, for $m,n \in {\mathbb N}$,
$$
X^mY^n=\sum_{r=0}^m \sum_{t=0}^n   \binom{m}{r}\binom{n}{t} \left\{ \sum_{\ell=1}^n \binom{r-t+\ell}{\ell}A(r,n-\ell) \right\} Y^{n-t}X^{m-r},
$$
where the $A(n,k)$ denotes the Eulerian numbers.
\end{proposition}
\begin{proof}
First observe that $XY-YX=I+Y$ implies $X(Y+1)=(Y+1)(X+1)$; here and in the following we identify $kI$ with $k$. An induction over $m$ and then over $k$ shows $X^m(Y+1)^k=(Y+1)^k(X+k)^m$. Writing $Y=(Y+1)-1$, one has
\begin{eqnarray*}
X^mY^n&=&\sum_{k=0}^n (-1)^{n-k}\binom{n}{k}(Y+1)^k(X+k)^m \\
&=&  \sum_{i=0}^n \sum_{j=0}^m \sum_{k=i}^n (-1)^{n-k}\binom{n}{k} \binom{k}{i}\binom{m}{j}k^{m-j} Y^{i} X^j \\ &=&  \sum_{i=0}^n \sum_{j=0}^m \binom{n}{i} \binom{m}{j} Y^{i}X^j \left\{ \sum_{k=i}^n (-1)^{n-k}\binom{n-i}{k-i}k^{m-j}\right\},
\end{eqnarray*}
where we expanded in the second equation the binomials and switched the order of summation, and used in the third equation that $\binom{n}{k} \binom{k}{i}=\binom{n}{i} \binom{n-i}{k-i}$. Letting $r=m-j$ and $t=n-i$, and then $\binom{t}{k-n+t}=\binom{t}{n-k}$, we obtain
$$
X^mY^n=  \sum_{t=0}^n \sum_{r=0}^m \binom{n}{t} \binom{m}{r} Y^{n-t}X^{m-r} \left\{ \sum_{k=n-t}^n (-1)^{n-k}\binom{t}{n-k}k^{r}\right\}. 
$$
The sum $\sum_{k=n-t}^n (-1)^{n-k}\binom{t}{n-k}k^{r}$ equals the expression $q_n(r,t)$ considered by Yeliussizov \cite{DY2019} who showed that $q_n(r,t)=\sum_{\ell=1}^n \binom{r-t+\ell}{\ell}A(r,n-\ell)$  with the Eulerian numbers $A(r,n-\ell)$ counting the number of permutations of $(1,\ldots,r)$ with $(n-\ell)$ descents.
\end{proof}
Comparing the two expressions for the normal ordered form of $X^mY^n$, we obtain the identity
$$
\sum_{\pi \in {\mathcal C}_t(m)} \sum_{\rho \in {\mathcal WC}_{t+1}(r-t || \pi)} \prod_{j=1}^{t+1} \binom{m_{j}(\pi)}{\rho_{j}}(n+1-j)^{\rho_{j}} =\frac{1}{t!}\sum_{\ell=1}^{n} \binom{m}{r}\binom{r-t+\ell}{\ell}A(r,n-\ell). 
$$

\section{The Binomial formula in the $q$-deformed generalized Ore algebra}\label{Binomial}

In this section, we consider the binomial $(X+Y)^m$ where $X$ and $Y$ satisfy \eqref{CROre}. For more on the binomial formula in noncommuting variables see \cite{TMMS2011} or the review \cite{MS2021} and the references given therein. Another recent approach can be found in \cite{HJYZ2023}, see also \cite{B2024}. 

We will expand the binomial into a sum over all binary words over the alphabet $\{X,Y\}$ of length $m$, list them with the help of partitions, and then use Theorem~\ref{THNOORE} for each word appearing. 

Let us introduce some basic notations for partitions following \cite{TMMS2011}. Recall that a {\em partition} $\lambda=(\lambda_1,\lambda_2,\ldots,\lambda_k) \equiv \lambda_1\lambda_2\ldots\lambda_k$ is a weakly decreasing sequence of positive integers. We denote the sum of the parts of $\lambda$ by $|\lambda|$, that is, $|\lambda|=\sum_{i=1}^k\lambda_i$. For a partition $\lambda$, the {\em Young diagram} $Y_\lambda$ of shape $\lambda$ is a left-justified diagram of $|\lambda|$ boxes, with $\lambda_i$ boxes in the $i$-th column. We denote the set of all Young diagrams that are contained in a $ \ell \times k$ box by ${\mathcal I}_{\ell,k}$. Define ${\mathcal I}_{m}=\bigcup_{k=0}^m {\mathcal I}_{m-k,k}$. Note that ${\mathcal I}_{m-k,k}$ contains $\binom{m}{k}$ elements, hence $|{\mathcal I}_m|=2^m$. Let us recall the following result \cite[Proposition 3.3]{TMMS2018}.
\begin{theorem}\label{thr1}
There exists a bijection between the Young diagrams in the set $\mathcal{I}_m$ and the binary words over the alphabet $\{X,Y\}$ of length $m$.
\end{theorem}
We recall the construction of this bijection. Let $\pi=\pi_1\pi_2\cdots\pi_m$ be any word over the alphabet $\{X,Y\}$ of length $m$. Let $\ell$ be the number of occurrences of the letter $X$ in $\pi$. Then the number of occurrences of the letter $Y$ in $\pi$ is given by ($m-\ell$). The numbers $\lambda_j$, for $1\leq j \leq \ell$, are defined by
\begin{align}\label{eqbij1}
\pi=Y^{m-\ell-\lambda_1}XY^{\lambda_1-\lambda_2}\cdots XY^{\lambda_{\ell-1}-\lambda_\ell}XY^{\lambda_\ell}.
\end{align}
By construction, $\lambda_\ell\geq 0$, $\lambda_j \geq \lambda_{j+1}$ (for $j=1,\ldots , \ell-1$), and $\lambda_1 \leq (m-\ell$), thus $\lambda=(\lambda_1,\lambda_2,\ldots,\lambda_\ell)\equiv \lambda_1\lambda_2\cdots\lambda_\ell$ is a partition with $Y_\lambda\in\mathcal{I}_{m-\ell,\ell}$. On the other hand, if $Y_\lambda$ is a Young diagram in $\mathcal{I}_{m-\ell,\ell}\subset\mathcal{I}_m$, then we define the word $\pi=\pi_1\pi_2\cdots\pi_m$ by \eqref{eqbij1}. Thus, $\pi$ is a binary word over the alphabet $\{X,Y\}$ of length $m$ with exactly $\ell$ letters of $X$ and exactly ($m-\ell$) letters of $Y$. Hence, we can map the set of all binary words over the alphabet $\{X,Y\}$ of length $m$ with exactly $\ell$ letters of $X$ and exactly ($m-\ell$) letters of $Y$ bijectively to the set $\mathcal{I}_{m-\ell,\ell}$. Clearly, this is a bijection between the set of all binary words over the alphabet $\{X,Y\}$ of length $m$ and the set $\mathcal{I}_m$.

Now, by expanding the binomial $(X+Y)^m$, we get a sum over all binary words of length $m$ in $X$ and $Y$. Thus, by Theorem~\ref{thr1} and, in particular \eqref{eqbij1}, we can write
\begin{equation}\label{binomgen}
(X+Y)^m=\sum_{\ell=0}^m \sum_{\lambda \in \mathcal{I}_{m-\ell,\ell}}Y^{m-\ell-\lambda_1}XY^{\lambda_1-\lambda_2}\cdots XY^{\lambda_{\ell-1}-\lambda_\ell}XY^{\lambda_\ell},
\end{equation}
where we write here -- and in the following -- with a slight abuse of notation $\lambda \in \mathcal{I}_{m-\ell,\ell}$ when we have $Y_{\lambda} \in \mathcal{I}_{m-\ell,\ell}$. Each summand is a word in $X$ and $Y$ of the form considered in Theorem~\ref{THNOORE}. Let us write, for $\lambda \in \mathcal{I}_{m-\ell,\ell}$,
$$
\omega_{\lambda}=XY^{\lambda_1-\lambda_2}\cdots XY^{\lambda_{\ell-1}-\lambda_\ell}XY^{\lambda_\ell}.
$$
Then we can write \eqref{binomgen} in the equivalent brief form
\begin{equation}\label{binomgen2}
(X+Y)^m=\sum_{\ell=0}^m \sum_{\lambda \in \mathcal{I}_{m-\ell,\ell}}Y^{m-\ell-\lambda_1}\omega_{\lambda},
\end{equation}
and it is clear that the main work consists in normal ordering $\omega_{\lambda}$. Note that all the considerations up to now were independent from the commutation relation holding for $X$ and $Y$. 

Now, we use the commutation relation. The words $\omega_{\lambda}$ have a special structure when compared to the general form considered in Theorem~\ref{THNOORE}. Here all $m_j=1$, for $j=1,\ldots,\ell$, thus $|{\bf m}|=|(1,1,\ldots,1)|=\ell$. Similarly, $|{\bf n}|=|(\lambda_1-\lambda_2,\lambda_2-\lambda_3,\ldots,\lambda_{\ell-1}-\lambda_{\ell},\lambda_{\ell} )|=\lambda_1$. 
Thus, \eqref{NOORE2} implies 
$$
\omega_{\lambda} = \sum_{r=0}^{\ell }\sum_{t=0}^{\min(\lambda_1, r)} m_{t,r-t}(B_{\omega_{\lambda}};q) Y^{\lambda_1-t}X^{\ell-r}. 
$$
Inserting this into \eqref{binomgen2} and switching to $k=\ell - r$ yields
$$
(X+Y)^m=\sum_{\ell=0}^m \sum_{\lambda \in \mathcal{I}_{m-\ell,\ell}}\sum_{k=0}^{\ell }\sum_{t=0}^{\min(\lambda_1, \ell - k)} m_{t,\ell -k-t}(B_{\omega_{\lambda}};q) Y^{m-\ell-t}X^{k}.
$$
Switching the order of summation, this shows the following result.
\begin{theorem}\label{THOREBINOM} Let $X$ and $Y$ satisfy the commutation relation \eqref{CROre} of the $q$-deformed generalized Ore algebra ${\mathcal O}_{\mu,\nu}(q)$. Then one has, for $m\in \mathbb{N}$, the normal ordered binomial formula,
\begin{equation}\label{OREBINOM}
 (X+Y)^m=\sum_{k=0}^{m}\sum_{\ell=k}^m \sum_{t=0}^{\ell - k} {\mathfrak O}_{m,k,\ell,t}(q) \, Y^{m-\ell-t}X^{k},
\end{equation}
where the combinatorial coefficients ${\mathfrak O}_{m,k,\ell,t}(q)$ are defined by
$$
{\mathfrak O}_{m,k,\ell,t}(q)=\sum_{\lambda \in \mathcal{I}_{m-\ell,\ell}} m_{t,\ell -k-t}(B_{\omega_{\lambda}};q).
$$
\end{theorem}

\begin{example} Consider the quantum plane ${\mathcal O}_{0,0}(q)$. From the definition, it is clear that only $m_{0,0}(B_{\omega_{\lambda}};q)=q^{|\lambda |}$ contributes, i.e., $t=0$ and $\ell = k$. Using $\sum_{\lambda \in \mathcal{I}_{m-k,k}}q^{|\lambda|}=\binom{m}{k}_q $, \eqref{OREBINOM} yields 
$$
(X+Y)^m=\sum_{k=0}^m  \sum_{\lambda \in \mathcal{I}_{m-k,k}}q^{|\lambda|} \, Y^{m-k}X^{k}=\sum_{k=0}^m \binom{m}{k}_q \, Y^{m-k}X^{k},
$$
i.e., the $q$-binomial formula due to Potter \cite{HP1950} and Schützenberger \cite{MS1953}.
\end{example}

\section{Extension to the $q$-deformed polynomial Weyl algebra}\label{Polynomial}
Varvak \cite{AV2005} gave an interpretation of the normal ordering coefficients of words in $X$ and $Y$ satisfying $XY=YX+hY^s$ (with $s\in \mathbb{N}_0$) in terms of the $s$-rook numbers introduced by Goldman and Haglund \cite{JGJH2000}. This was extended to the $q$-deformed case by Celeste et al. \cite{CCG2017}. We briefly recall this.

Let $i\in \mathbb{N}_0$ and let $B$ be a Ferrers board. The {\em $i$-row creation rule} of Goldman and Haglund \cite{JGJH2000} for rook placements means that each time we choose to place a rook, we create $i$ new rows to the left of the box. The rooks are placed on the Board $B$ going from right to left. When creating $i$ new rows, we interpret the resulting top row as the row belonging to the rook just placed while the $i$ rows beneath it as new rows. See Figure~\ref{FigFile2} for an example with a 1-row creation rule. Note that we indicated in the rook which row creation rule was used.

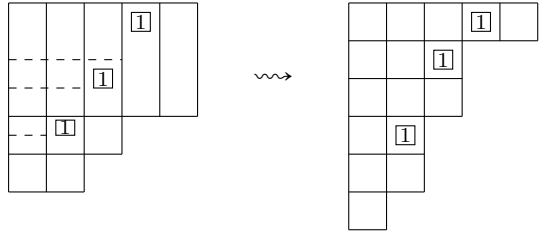
\begin{figure}[h]
\centering
\begin{tikzpicture}

\draw (0,2.5) -- (2.5,2.5);
\draw (0,1) -- (2.5,1);
\draw (0,0.5) -- (1.5,0.5);
\draw (0,0) -- (1,0);

\draw (0,0) -- (0,2.5);
\draw (0.5,0) -- (0.5,2.5);
\draw (1,0) -- (1,2.5);
\draw (1.5,0.5) -- (1.5,2.5);
\draw (2,1) -- (2,2.5);
\draw (2.5,1) -- (2.5,2.5);

\draw[dashed] (0,1.75) -- (1.5,1.75);
\draw[dashed] (0,1.375) -- (1,1.375);
\draw[dashed] (0,0.75) -- (0.5,0.75);

\draw (1.625,2.125) rectangle (1.875, 2.375);
\node at (1.75, 2.25) {\footnotesize 1};

\draw (1.125,1.375) rectangle (1.375, 1.625);
\node at (1.25, 1.5) {\footnotesize 1};

\draw (0.625,0.75) rectangle (0.875, 0.95);
\node at (0.75, 0.86) {\footnotesize 1};

\node at (3.5,1.5) {$\rightsquigarrow$};
\draw (4.5,2.5) -- (7,2.5);
\draw (4.5,2) -- (7,2);
\draw (4.5,1.5) -- (6,1.5);
\draw (4.5,1) -- (6,1);
\draw (4.5,0.5) -- (5.5,0.5);
\draw (4.5,0) -- (5.5,0);
\draw (4.5,-0.5) -- (5,-0.5);

\draw (4.5,-0.5) -- (4.5,2.5);
\draw (5,-0.5) -- (5,2.5);
\draw (5.5,0) -- (5.5,2.5);
\draw (6,1) -- (6,2.5);
\draw (6.5,2) -- (6.5,2.5);
\draw (7,2) -- (7,2.5);

\draw (6.125,2.125) rectangle (6.375, 2.375);
\node at (6.25, 2.25) {\footnotesize 1};

\draw (5.625,1.625) rectangle (5.875, 1.875);
\node at (5.75, 1.75) {\footnotesize 1};

\draw (5.125,0.625) rectangle (5.375, 0.875);
\node at (5.25, 0.75) {\footnotesize 1};

\end{tikzpicture}
\caption{A Ferrers board with a placement of 3 rooks satisfying the 1-row creation rule (left) and its equivalent visualization (right).}\label{FigFile2}
\end{figure} 

Following Goldman and Haglund \cite{JGJH2000}, we define the {\em $i$-rook numbers} as follows.
\begin{definition} Let $i\in \mathbb{N}_0$. Given a Ferrers board $B$, the {\em $k$-th $i$-rook number} $r_k^{(i)}(B)$ is the number of ways to place $k$ non-attacking rooks on the board $B$ going from right to left, creating $i$ new rows to the left of each rook.  
\end{definition}
For $i=0$, one recovers the conventional rook numbers. For $i=1$, each time we place a rook, we create a new row to the left. Thus, for the next rook to place, in each column to the left one has the same amount of boxes to choose from as the column had before. In other words, the placement of $k$ rooks with 1-row creation rule corresponds to a $k$-file placement. Thus,  
\begin{equation}\label{RookRed}
r_k^{(0)}(B)=r_k(B), \,\,\,\, r_k^{(1)}(B)=f_k(B). 
\end{equation}
\begin{definition}A rook satisfying the $i$-row creation rule will be called a {\em rook of weight $i$}. In particular, a conventional rook will be called a rook of weight 0, while a file will be called a rook of weight 1. 
\end{definition}
Using this terminology, we say that 3 rooks of weight 1 have been placed on the board in Figure~\ref{FigFile2}.

Let us consider a word $\mathrm{w}_{{\bf m}, {\bf n}}=Y^{n_r}X^{m_r}\cdots Y^{n_1}X^{m_1}$ where $XY=YX+hY^s$. To this word we associate a Ferrers board $B_{\mathrm{w}_{{\bf m}, {\bf n}}}$ as described above. In this situation, Varvak \cite[Theorem 7.1]{AV2005} showed that (after correcting a typo) the normal ordered form of $\mathrm{w}_{{\bf m}, {\bf n}}$ is given in terms of $s$-rook numbers by
\begin{equation}\label{RookVarvak0}
\mathrm{w}_{{\bf m}, {\bf n}}=\sum_{j=0}^{|{\bf m}|-m_1} h^j r_j^{(s)}(B_{\mathrm{w}_{{\bf m}, {\bf n}}}) Y^{|{\bf n}|+(s-1)j}X^{|{\bf m}|-j}.
\end{equation}
Recall the procedure from Section~\ref{Normal} in the new terminology: For the $q$-deformed generalized Ore algebra ${\mathcal O}_{\alpha_0,\alpha_1}(q)$ where $XY=qYX+\alpha_0 I + \alpha_1 Y$ one has in the process of normal ordering for each subword $XY$ to decide between the following three possibilities, corresponding to the terms on the right-hand side of the commutation relation:
\begin{itemize}
\item leave it empty (corresponding to $XY \rightsquigarrow q YX$), or
\item place a rook of weight 0 (corresponding to $XY \rightsquigarrow \alpha_0 I$), or
\item place a rook of weight 1 (corresponding to $XY \rightsquigarrow \alpha_1 Y$).
\end{itemize}
The above-mentioned result of Varvak shows that the term $\alpha_s Y^s$ on the right-hand side of the commutation relation corresponds to the placement of a rook of weight $s$ (corresponding to $XY \rightsquigarrow \alpha_s Y^s$). Thus, to describe the normal ordering in $\mathcal{A}_{f}(q)$ where one has
\begin{equation}\label{WeylPoly}
XY=qYX+\alpha_0 I+ \alpha_1 Y+\alpha_2 Y^2+\cdots +\alpha_s Y^s,
\end{equation}
we have to bring these observations together. Roughly speaking, each time we come across a subword $XY$ we can choose between $(s+2)$ possibilities: Either do not place a rook (corresponding to $XY \rightsquigarrow q YX$), or place a rook of weight $k$ (corresponding to $XY \rightsquigarrow \alpha_k Y^k$), where $k=0,1,\ldots,s$. Then one has to sum over all possibilities of placing rooks of weights $k$ with $k=0,\ldots,s$ (or not). We now formalize these notions. Let us introduce, following Goldman and Haglund \cite{JGJH2000}, the following notion.

\begin{definition}\label{MixedPlGen} Let $B$ be a Ferrers board. Let $s \in \mathbb{N}_0$, and let $k_j \in \mathbb{N}_0$, for $j=0,\ldots , s $. A {\em mixed rook placement of type ${\bf k} \equiv (k_0,\ldots,k_s)$ on $B$} is a non-attacking placement $\phi$ of $k_j$ rooks with weights $j$, $j=0,\ldots ,s$, on $B$, going from right to left, creating $j$ rows to the left of each rook of weight $j$. The set of all mixed rook placements of type ${\bf k}$ on $B$ will be denoted by ${\mathcal M}_{{\bf k}}(B)$ or ${\mathcal M}_{k_0,\ldots,k_s}(B)$.
\end{definition}
\begin{example} In Figure~\ref{FigMPL}, a mixed rook placement of type $(1,1,0,1)$ on the staircase board $J_5$ is shown (in the equivalent visualization like on the right-hand side of Figure~\ref{FigFile2}). 

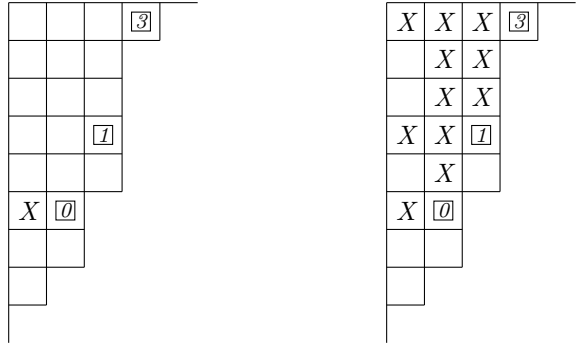
\begin{figure}[h]
\centering
\begin{tikzpicture}

\draw (4.5,2.5) -- (7,2.5);
\draw (4.5,2) -- (6.5,2);
\draw (4.5,1.5) -- (6,1.5);
\draw (4.5,1) -- (6,1);
\draw (4.5,0.5) -- (6,0.5);
\draw (4.5,0) -- (6,0);
\draw (4.5,-0.5) -- (5.5,-0.5);
\draw (4.5,-1) -- (5.5,-1);
\draw (4.5,-1.5) -- (5,-1.5);

\draw (4.5,-2) -- (4.5,2.5);
\draw (5,-1.5) -- (5,2.5);
\draw (5.5,-1) -- (5.5,2.5);
\draw (6,0) -- (6,2.5);
\draw (6.5,2) -- (6.5,2.5);

\draw (6.125,2.125) rectangle (6.375, 2.375);
\node at (6.25, 2.25) {\footnotesize 3};

\draw (5.625,0.625) rectangle (5.875, 0.875);
\node at (5.75, 0.75) {\footnotesize 1};

\draw (5.125,-0.375) rectangle (5.375, -0.125);
\node at (5.25, -0.25) {\footnotesize 0};

\node at (4.75,-0.25) {$\mbox{X}$};


\draw (9.5,2.5) -- (12,2.5);
\draw (9.5,2) -- (11.5,2);
\draw (9.5,1.5) -- (11,1.5);
\draw (9.5,1) -- (11,1);
\draw (9.5,0.5) -- (11,0.5);
\draw (9.5,0) -- (11,0);
\draw (9.5,-0.5) -- (10.5,-0.5);
\draw (9.5,-1) -- (10.5,-1);
\draw (9.5,-1.5) -- (10,-1.5);

\draw (9.5,-2) -- (9.5,2.5);
\draw (10,-1.5) -- (10,2.5);
\draw (10.5,-1) -- (10.5,2.5);
\draw (11,0) -- (11,2.5);
\draw (11.5,2) -- (11.5,2.5);

\draw (11.125,2.125) rectangle (11.375, 2.375);
\node at (11.25, 2.25) {\footnotesize 3};

\draw (10.625,0.625) rectangle (10.875, 0.875);
\node at (10.75, 0.75) {\footnotesize 1};

\draw (10.125,-0.375) rectangle (10.375, -0.125);
\node at (10.25, -0.25) {\footnotesize 0};

\node at (9.75, 2.25) {$\mbox{X}$};
\node at (9.75, 0.75) {$\mbox{X}$};
\node at (9.75,-0.25) {$\mbox{X}$};

\node at (10.25,2.25) {$\mbox{X}$};
\node at (10.25,1.75) {$\mbox{X}$};
\node at (10.25,1.25) {$\mbox{X}$};
\node at (10.25,0.75) {$\mbox{X}$};
\node at (10.25,0.25) {$\mbox{X}$};

\node at (10.75,2.25) {$\mbox{X}$};
\node at (10.75,1.75) {$\mbox{X}$};
\node at (10.75,1.25) {$\mbox{X}$};

\end{tikzpicture}
\caption{The staircase board $J_5$ with a mixed rook placement of type $(1,1,0,1)$ (left), and where the boxes are marked according to their class (right).}\label{FigMPL}
\end{figure} 
\end{example}
Note that in the definition of a mixed rook placement no particular order of the rooks of different weights is assumed. A mixed rook placement of type ${\bf k} = (0,\ldots,0,k_{\ell},0,\ldots,0)$ is a non-attacking placement of $k_{\ell}$ rooks of weight $\ell$ and will also be called a {\em pure rook placement of degree $\ell$}. Thus, 
\begin{equation}\label{RedVarv}
|{\mathcal M}_{0,\ldots,0,k_{\ell},0,\ldots,0}(B)|=r_{k_{\ell}}^{(\ell)}(B).
\end{equation}
By construction, the above placements generalize the mixed placements of rooks and files we considered in Section~\ref{Normal}. Namely, a mixed rook placement of type $(k_0,k_1,0,\ldots,0)$ is a non-attacking placement of $k_{0}$ rooks and $k_1$ files, i.e., 
\begin{equation}\label{RedOre}
{\mathcal M}_{k_0,k_1,0,\ldots,0}(B)={\mathcal M}_{k_0,k_1}(B).
\end{equation}

Let $B$ be a Ferrers board and let $\phi \in \mathcal{M}_{k_0,\ldots,k_s}(B)$. We partition the boxes of $B$ into $(s+3)$ classes:
\begin{enumerate}
\item A box is called a {\em rook box of weight $j$} if a rook of weight $j$ is placed in it, for $j=0,\ldots,s$,
\item A box is called a {\em cancelled box}, if it is neither a rook box (irrespective of weight), and it is lying above a rook in the same column or to the left of a rook in the same row.
\item All remaining boxes are {\em empty boxes}. 
\end{enumerate}
Let us abbreviate the set of parameters by ${\bf \alpha}=(\alpha_0,\dots,\alpha_s)$ and use multi-index notation,
$$
{\bf  \alpha}^{{\bf k}}=  \prod_{j=0}^s \alpha_j^{k_j}.
 $$
\begin{definition}
The {\em $q$-weight of the mixed placement} $\phi \in \mathcal{M}_{{\bf k}}(B)$ is defined to be 
\begin{equation}
w_{{\bf \alpha};q}(B,\phi)=q^{\# \mbox{empty boxes}} \prod_{j=0}^s \alpha_j^{\# \mbox{rook boxes of weight $j$}}= {\bf  \alpha}^{{\bf k}} q^{\# \mbox{empty boxes}}.
\end{equation}
\end{definition}
\begin{remark}
The definition of the $q$-weights reduces for $\phi$ of type $(k_0=\mu,k_1=\nu,0,\ldots,0)$ to the definition of $w_{\mu,\nu;q}(B,\phi)$ for the $q$-deformed generalized Ore algebra. However, it seems that our definition of the classes of the boxes does not coincide with the one given there. This impression is false and the reason is that we use here for a placement of a file the 1-rule creation rule, where we add a new row to the left of the file and interpret the row of the file as filled. In Section~\ref{Normal}, we did not create a new row, but allowed placements of files and rooks to the left of the file in that row.
\end{remark}

For example, the mixed rook placement of type $(1,1,0,1)$ shown in Figure~\ref{FigMPL} has $q$-weight $q^7 \alpha_0\alpha_1\alpha_3$. Let us turn to the connection to normal ordering. The staircase board $J_5$ is derived from the word $(YX)^5$. The particular sequence of contractions and commutations described by the mixed rook placement shown in Figure~\ref{FigMPL} contains one placement of a rook of weight 0 ($XY \rightsquigarrow \alpha_0 I$), one placement of a rook of weight 1 ($XY \rightsquigarrow \alpha_1Y$), and one placement of a rook of weight 3 ($XY \rightsquigarrow \alpha_3 Y^3$). Thus, the resulting normal ordered expression is $q^7 \alpha_0\alpha_1\alpha_3 Y^{5-1+0+2}X^{5-1-1-1}=q^7 \alpha_0\alpha_1\alpha_3 Y^{6}X^{2}$. In general, we have the following result.
\begin{lemma} Let a word $\mathrm{w}_{{\bf m}, {\bf n}}=Y^{n_r}X^{m_r}\cdots Y^{n_1}X^{m_1}$ in ${\mathcal A}_f(q)$ be given where $f(Y)=\sum_{j=0}^s \alpha_jY^j$. Let $\phi \in {\mathcal M}_{{\bf k}}(B_{\mathrm{w}_{{\bf m}, {\bf n}}})$ be a mixed rook placement of type ${\bf k}=(k_0,\ldots,k_s)$. Then the corresponding summand has the normal ordered form
\begin{equation}\label{WORD}
q^{\ell} {\bf  \alpha}^{{\bf k}} \, Y^{|{\bf n}|+\sum_{j=0}^s (j-1)k_j}X^{|{\bf m}|-|{\bf k}|},
\end{equation}
where $|{\bf k}|=k_0 +\cdots + k_s$ and the exponent $\ell$ depends on the concrete placement of the rooks on $B_{\mathrm{w}_{{\bf m}, {\bf n}}}$.
\end{lemma}
\begin{proof}
Each placement of a rook of weight $j$ gives a factor of $\alpha_j$ and reduces the number of letters $X$ by one and changes the number of letters $Y$ by $(j-1)$. Thus, the placement of $k_j$ rooks of weight $j$ gives the factor $\alpha_j^{k_j}$, reduces the number of letters $X$ by $k_j$ and changes the number of letters $Y$ by $(j-1)k_j$. Considering all different weights yields the right-hand side up to the power of $q$. This comes from the number of empty boxes and depends on the concrete placement of the rooks.  
\end{proof}
Note that if $\phi$ is pure of degree $s$, i.e., one has a placement of $j$ rooks of weight $s$, the expression on the right-hand side of \eqref{WORD} reduces  for $q=1$ to $\alpha_s^{j} Y^{|{\bf n}|+(s-1)j}X^{|{\bf m}|-j}$, recovering the structure of the summands on the right-hand side of \eqref{RookVarvak0}.  

Before we can collect all summands over all possibilities of non-attacking mixed rook placements, we extend the $q$-deformed mixed placement numbers from Definition~\ref{DefOREMPN} to the present situation. 
\begin{definition}\label{DefMPN}
The {\em $q$-deformed mixed placement numbers of type ${\bf k}=(k_0,\ldots,k_s)$} are defined by 
\begin{equation}
m_{{\bf k}}(B;q)\equiv m_{{\bf k}}^{\bf  \alpha}(B;q)=\sum_{\phi \in {\mathcal M}_{{\bf k}}(B)} w_{{\bf \alpha};q}(B,\phi)=  {\bf  \alpha}^{{\bf k}} \sum_{\phi \in {\mathcal M}_{{\bf k}}(B)} q^{\# \mbox{empty boxes}}.
\end{equation}
\end{definition}
For $q=1$, this reduces to $m_{{\bf k}}(B;1)={\bf  \alpha}^{{\bf k}}|{\mathcal M}_{{\bf k}}(B)|$. 
Now, we have to sum over all possible types of placements. For $k,s\in \mathbb{N}_0$, we let
$$
{\mathcal WC}_{s+1}(k)=\left\{{\bf k} \equiv (k_0,\ldots,k_s) \in \mathbb{N}_0^{s+1} \, | \, |{\bf k} | \equiv \sum_{j=0}^s k_j = k \right\},
$$
i.e., the set of weak compositions of $k$ with $(s+1)$ parts. 
\begin{theorem}\label{THNOWEYLPOLY} Let $\mathrm{w}_{{\bf m}, {\bf n}}=Y^{n_r}X^{m_r}\cdots Y^{n_1}X^{m_1}$ be a word in ${\mathcal A}_f(q)$ where $f(Y)=\sum_{j=0}^s \alpha_jY^j$. Then the normal ordered form of $\mathrm{w}_{{\bf m}, {\bf n}}$ is given by
\begin{equation}\label{NOWEYLPOLY}
\mathrm{w}_{{\bf m}, {\bf n}}=\sum_{k=0}^{|{\bf m}|-m_1} \sum_{{\bf k} \in {\mathcal WC}_{s+1}(k)} m_{{\bf k}}(B_{\mathrm{w}_{{\bf m}, {\bf n}}};q) \, Y^{|{\bf n}|+\sum_{j=0}^s (j-1)k_j} X^{|{\bf m}|- k}.
\end{equation}
\end{theorem}
\begin{proof}
The normal ordered form of $\mathrm{w}_{{\bf m}, {\bf n}}$ is given as the sum over all possible non-attacking mixed rook placements, where for each placement we have a normal ordered summand. To take care of all such mixed placements, we sum over all possible types of placements, explaining the double sum in \eqref{NOWEYLPOLY}. For each placement of a fixed type, the structure of the normal ordered summand was determined in \eqref{WORD}. The sum over all placements of a fixed type gives the combinatorial prefactor $m_{{\bf k}}(B_{\mathrm{w}_{{\bf m}, {\bf n}}};q) $. 
\end{proof}
Let us consider the case $s=1$ of \eqref{NOWEYLPOLY}. Then ${\bf k}=(k_0,k_1) \in {\mathcal WC}_{2}(k)$ means that $k_0+k_1=k$. Letting $t \equiv k_0$, one has ${\bf k}=(t,k-t)$ and $\sum_{j=0}^1(j-1)k_j=-k_0=-t$. Hence, \eqref{NOWEYLPOLY} reduces to
$$
\mathrm{w}_{{\bf m}, {\bf n}}=\sum_{k=0}^{|{\bf m}|-m_1} \sum_{t=0}^{\min(|{\bf n}|,k)} m_{t,k-t}(B_{\mathrm{w}_{{\bf m}, {\bf n}}};q) \, Y^{|{\bf n}|-t} X^{|{\bf m}|- k},
$$ 
recovering \eqref{NOORE2}.
\begin{corollary} Let $XY=qYX+\alpha_{\ell}Y^{\ell}$. Then \eqref{NOWEYLPOLY} reduces to the $q$-deformed version of \eqref{RookVarvak0} given by Celeste et al. \cite{CCG2017}.
\end{corollary}
\begin{proof}
In the case under consideration, all rook placements are pure of degree $\ell$, i.e., have the form ${\bf k} =(0,\ldots,0,k_{\ell},0,\ldots,0)$. From the set of weak compositions of $k$ exactly one member remains, $(0,\ldots,0,k_{\ell}=k,0,\ldots,0)$. For $q=1$, we observed above that, in general, $m_{{\bf k}}(B_{\mathrm{w}_{{\bf m}, {\bf n}}};1)={\bf  \alpha}^{{\bf k}}|{\mathcal M}_{{\bf k}}(B_{\mathrm{w}_{{\bf m}, {\bf n}}})|$ reducing for the  pure case to $\alpha_{\ell}^{k_{\ell}}|{\mathcal M}_{0,\ldots,0,k_{\ell},0,\ldots,0}(B_{\mathrm{w}_{{\bf m}, {\bf n}}})|$. Thus, using \eqref{RedVarv}, we find
$$
\mathrm{w}_{{\bf m}, {\bf n}}=\sum_{k=0}^{|{\bf m}|-m_1} \alpha_{\ell}^{k} r_{k}^{(\ell)}(B_{\mathrm{w}_{{\bf m}, {\bf n}}}) Y^{|{\bf n}|+(\ell-1)k} X^{|{\bf m}|- k},
$$
which is exactly \eqref{RookVarvak0}. 
A pure placement of degree $\ell$ of $k$ rooks on $B$ is exactly a rook placement of $k$ rooks satisfying the $\ell$-row creation rule, thus ${\mathcal M}_{0,\ldots,0,k_{\ell},0,\ldots,0}(B)={\mathcal R}_{\ell}(B,k_{\ell})$ in the notation of \cite{CCG2017}. Hence, for $q\neq 1$, we obtain from Definition~\ref{DefMPN} in the pure case 
\begin{equation}\label{RedPure}
m_{0,\ldots,0,k_{\ell},0,\ldots,0}(B_{\mathrm{w}_{{\bf m}, {\bf n}}};q)= \alpha_{\ell}^{k_{\ell}} \sum_{\phi \in {\mathcal R}_{\ell}(B_{\mathrm{w}_{{\bf m}, {\bf n}}},k_{\ell})} q^{\# \mbox{empty boxes}}\equiv R_{\ell, \alpha_{\ell};q}[B_{\mathrm{w}_{{\bf m}, {\bf n}}},k_{\ell}],
\end{equation}
where the second equation is the definition of the $q$-deformed $\ell$-rook numbers according to \cite{CCG2017}. Thus, 
$$
\mathrm{w}_{{\bf m}, {\bf n}}=\sum_{k=0}^{|{\bf m}|-m_1} R_{\ell, \alpha_{\ell};q}[B_{\mathrm{w}_{{\bf m}, {\bf n}}},k] Y^{|{\bf n}|+(\ell-1)k} X^{|{\bf m}|- k},
$$ 
which is exactly the $q$-deformed generalization of \eqref{RookVarvak0} given by Celeste et al. \cite{CCG2017}.
\end{proof}
\begin{remark}
As mentioned in the introduction, some basic normal ordering results for ${\mathcal A}_f(q)$ were derived in \cite{TMMS2011}. For $q=1$, several ordering results for ${\mathcal A}_f(1)$ can be found in \cite{BLO2013,BLO2015,BLO2015a}, see also \cite{BRLO2020,Lopes2024}.
\end{remark}
Recall that the normal ordering coefficients of the word $(YX)^n$ have been considered in different cases:
\begin{itemize}
\item Weyl algebra ($\alpha_0=1$): Stirling numbers of the second kind, $S(n,k)$, see \eqref{Scherk}.
\item Shift algebra ($\alpha_1=1$): Unsigned Stirling numbers of  the first kind, $|s(n,k)|$, see \eqref{ShiftScherk}.
\item Generalized Weyl algebra ($\alpha_s=1$): Generalized Stirling numbers ${\mathfrak S}_{s;1}(n,k)$ \cite{TMMSMS2011,TMMSMS2012}.
\item Ore algebra ($\alpha_0=\alpha_1=1$): Ore-Stirling numbers, see \eqref{QOreStirlingDef}.
\end{itemize}
For all these numbers, a $q$-deformed version has been considered as well. In particular, the definition of the $q$-deformed Ore-Stirling numbers in \eqref{QOreStirlingDef} motivates the following definition.
\begin{definition} Let $X$ and $Y$ satisfy the commutation relation \eqref{WeylPoly} of the $q$-deformed polynomial Weyl algebra ${\mathcal A}_f(q)$. Then the {\em $q$-deformed polynomial Stirling numbers of degree $s$}, ${\mathscr S}_{{\bf \alpha};q}(n;j,k)$, are defined as normal ordering coefficients of $(YX)^n$ in ${\mathcal A}_f(q)$,
$$
(YX)^n=\sum_{j=0}^{ns}\sum_{k=1}^n {\mathscr S}_{{\bf \alpha};q}(n;j,k )Y^j X^{k}. 
$$
\end{definition}
To see that $j$ is at most $ns$ note that the maximum is reached when, starting from $Y^n$, there are $(n-1)$ placements of a rook of weight $s$, giving $Y^{n+(n-1)s}=Y^{ns}$. Before we can give a description of the ${\mathscr S}_{{\bf \alpha};q}(n;j,k)$, we introduce a partition of  ${\mathcal WC}_{s+1}(k)$ into smaller pieces. Let us introduce, for $t=-k,\ldots,0,\ldots, (s-1)k$, the sets
$$
{\mathcal WC}_{s+1}(k|t)=\left\{{\bf k} \in {\mathcal WC}_{s+1}(k) \, | \, \sum_{j=0}^s (j-1)k_j = t \right\}.
$$
The minimal value $t=-k$ is reached by the placement ${\bf k}=(k_0=k,0,\ldots,0)$ while the maximal value $t=(s-1)k$ is reached by the placement ${\bf k}=(0,0,\ldots,0,k_s=k)$. Thus, we have a partition
\begin{equation}\label{partition}
{\mathcal WC}_{s+1}(k)=\bigcup_{t=-k}^{(s-1)k} {\mathcal WC}_{s+1}(k|t).
\end{equation}
Now, we can formulate the following result. 
\begin{proposition} The $q$-deformed polynomial Stirling numbers of degree $s$ are given by 
\begin{equation}\label{PolyStir}
{\mathscr S}_{{\bf \alpha};q}(n;j,k)=\sum_{{\bf k} \in {\mathcal WC}_{s+1}(n-k|j-n)} m_{{\bf k}}(J_n;q).
\end{equation}
\end{proposition}
\begin{proof}
Specializing \eqref{NOWEYLPOLY} for the word $(YX)^n$ and switching to $\ell=n-k$, one finds
$$
(YX)^n=\sum_{\ell=1}^{n} \sum_{{\bf k} \in {\mathcal WC}_{s+1}(n-\ell)} m_{{\bf k}}(J_n;q) \, Y^{n+\sum_{j=0}^s (j-1)k_j} X^{\ell}.
$$
Using the partition \eqref{partition} and letting $j=n+t$, this yields
$$
(YX)^n=\sum_{\ell=1}^{n} \sum_{j=\ell}^{n+(s-1)(n-\ell)} \sum_{{\bf k} \in {\mathcal WC}_{s+1}(n-\ell|j-n)} m_{{\bf k}}(J_n;q) \, Y^{j} X^{\ell}.
$$
Comparing with the definition of ${\mathscr S}_{{\bf \alpha};q}(n;j,k)$ shows the assertion.  
\end{proof}
\begin{example}
Let $s=1$. In \eqref{PolyStir}, ${\bf k}\equiv (k_0,k_1) \in {\mathcal WC}_{2}(n-k|j-n)$ means that $k_0+k_1=n-k$ and $-k_0=j-n$. Thus, the sum reduces to one summand ${\bf k}=(n-j,j-k)$. Hence, 
$$
{\mathscr S}_{(\alpha_0,\alpha_1);q}(n;j,k)=m_{n-j,j-k}(J_n;q)=S_{\alpha_0,\alpha_1;q}(n;j,k),
$$ 
recovering the $q$-deformed Ore-Stirling numbers $S_{\alpha_0,\alpha_1;q}(n;j,k)$ according to \eqref{QOREEX}. 
\end{example}

One can extend Proposition~\ref{PROPQORESRec} to the $q$-deformed polynomial Stirling numbers as follows.
\begin{proposition}\label{THRECPOLY} The $q$-deformed polynomial Stirling numbers of degree $s$ satisfy the recurrence relation 
\begin{equation}\label{RECPOLY}
{\mathscr S}_{{\bf \alpha};q}(n+1;j,k)=q^{j-1}{\mathscr S}_{{\bf \alpha};q}(n;j-1,k-1)+\sum_{r=0}^s \alpha_r [j-r]_q {\mathscr S}_{{\bf \alpha};q}(n;j-r,k).
\end{equation}
\end{proposition}
\begin{proof}
The proof is combinatorial and very similar to the one of Proposition~\ref{PROPQORESRec}. Using \eqref{PolyStir} and the definition of the $m_{{\bf k}}(B;q)$, one has
$$
{\mathscr S}_{{\bf \alpha};q}(n+1;j,k)=\sum_{{\bf k} \in {\mathcal WC}_{s+1}(n+1-k|j-n-1)}{\bf  \alpha}^{{\bf k}} \sum_{\phi \in {\mathcal M}_{{\bf k}}(J_{n+1})} q^{\# \mbox{empty boxes}}.
$$ 
Thus, we have to consider all mixed placements $\phi \in \mathcal{M}_{{\bf k}}(J_{n+1})$. As above, we write in an informal fashion $J_{n+1}=C_n \oplus J_n$. For each type ${\bf k}$, there are $(s+2)$ types of placements of the rooks of the different weights on $J_{n+1}$:
\begin{itemize}
\item Type $I$: Place all rooks on $J_n$ and leave $C_n$ empty.
\item Type $II_r$: Place one rook of weight $r$ in $C_n$ and the remaining rooks on $J_{n}$, for $r=0,\ldots,s$.
\end{itemize} 
Thus, we can write
$$
\mathcal{M}_{{\bf k}}(J_{n+1})=\mathcal{M}_{{\bf k}}^{I}(J_{n+1})\cup\mathcal{M}_{{\bf k}}^{II_0}(J_{n+1})\cup\mathcal{M}_{{\bf k}}^{II_1}(J_{n+1})\cup \cdots \cup\mathcal{M}_{{\bf k}}^{II_s}(J_{n+1}).
$$
Let us start with placements of type $I$. Each such placement  $\phi  \in \mathcal{M}_{{\bf k}}^{I}(J_{n+1})$ corresponds to a placement $\phi' \in \mathcal{M}_{{\bf k}}(J_{n})$. In $C_n$, one has for a placement of type ${\bf k}=(k_0,\ldots,k_s)$ in total $n+\sum_{j=0}^s(j-1)k_j$ empty boxes. Here ${\bf k}\in {\mathcal WC}_{s+1}(n+1-k|j-n-1)$, i.e., $\sum_{j=0}^s(j-1)k_j=(j-n-1)$. Hence, there are $n+(j-n-1)=j-1$ empty boxes in $C_n$. If $\phi'$ has weight $\omega_{{\bf \alpha};q}(J_n,\phi')$ then $\omega_{{\bf \alpha};q}(J_{n+1},\phi)=q^{j-1}\omega_{{\bf \alpha}}(J_n,\phi')$. Thus, the contribution of all placements of type $I$ is given by
$$
\sum_{{\bf k} \in {\mathcal WC}_{s+1}(n-(k-1)|(j-1)-n)}{\bf  \alpha}^{{\bf k}} q^{j-1}\sum_{\phi' \in {\mathcal M}_{{\bf k}}(J_{n})} q^{\# \mbox{empty boxes}} =q^{j-1}{\mathscr S}_{{\bf \alpha};q}(n;j-1,k-1).
$$
Let us turn to placements of type $II_r$ (for $r=0,\ldots,s$). For ${\bf k}=(k_0,\ldots,k_s)$, let us denote 
$$
{\bf k}^{(r)}\equiv (k_0^{(r)},\ldots,k_{s}^{(r)})=(k_0-\delta_{r,0},\ldots,k_{s}-\delta_{r,s}), 
$$
i.e., the $r$-th entry is reduced by one. One has $\sum_{j=0}^s(j-1)k_{j}^{(r)}=\sum_{j=0}^s(j-1)k_{j}-(r-1)$. Thus, 
$$
{\bf k}\in {\mathcal WC}_{s+1}(n+1-k|j-n-1) \Longrightarrow {\bf k}^{(r)}\in {\mathcal WC}_{s+1}(n-k|j-n-r),
$$
and one has ${\bf  \alpha}^{{\bf k}}=\alpha_r{\bf  \alpha}^{{\bf k}^{(r)}}$. If $\phi' \in \mathcal{M}_{{\bf k}^{(r)}}(J_{n})$, then there are $n+\sum_{j=0}^s(j-1)k_{j}^{(r)}=j-r$ possible boxes to place the rook of weight $r$ on $C_n$ to get a placement $\phi \in \mathcal{M}_{{\bf k}}^{II_r}(J_{n+1})$. Placing the rook in the $\ell$-th possible row from above (where $\ell=1,\ldots,j-r$) will result in $(j-r-\ell)$ additional empty boxes. Thus, the sum of the weights of these $j-r$ placements $\phi$ is given by $\alpha_r (1+q+\cdots+q^{j-r-1})\omega_{{\bf \alpha};q}(J_n,\phi')=\alpha_r [j-r]_q \omega_{{\bf \alpha};q}(J_n,\phi')$. The contribution of all placements of type $II_r$ is, therefore, given by
$$
\sum_{{\bf k}^{(r)} \in {\mathcal WC}_{s+1}(n-k|j-r-n)} \alpha_r {\bf  \alpha}^{{\bf k}^{(r)}}[j-r]_q 
\sum_{\phi' \in {\mathcal M}_{{\bf k}^{(r)}}(J_{n})} q^{\# \mbox{empty boxes}}=\alpha_r [j-r]_q {\mathscr S}_{{\bf \alpha};q}(n;j-r,k).
$$
Summing over the types $I$ and $II_r$, for $r=0,\ldots,s$, yields the assertion.
\end{proof}

As extension of the $q$-deformed Ore-Lah numbers defined in\eqref{QOreLahDef}, we define the {\em $q$-deformed polynomial Lah numbers of degree $s$}, ${\mathscr L}_{{\bf \alpha};q}(n;j,k)$, as normal ordering coefficients of $(Y^2X)^n$ in ${\mathcal A}_f(q)$,
$$
(Y^2X)^n=\sum_{j=0}^{(s+1)n}\sum_{k=1}^{n} {\mathscr L}_{{\bf \alpha};q}(n;j,k )Y^j X^{k}. 
$$
Recalling that the word $(Y^2X)^n$ outlines the Lah board ${\mathcal L}_n$, one can show in the same way the following analogues of \eqref{PolyStir} and \eqref{RECPOLY}.
\begin{proposition} The $q$-deformed polynomial Lah numbers of degree $s$ are given by 
\begin{equation}\label{PolyLah}
{\mathscr L}_{{\bf \alpha};q}(n;j,k)=\sum_{{\bf k} \in {\mathcal WC}_{s+1}(n-k|j-2n)} m_{{\bf k}}({\mathcal L}_n;q).
\end{equation}
They satisfy the recurrence relation
\begin{equation}\label{PolyLahRec}
{\mathscr L}_{{\bf \alpha};q}(n+1;j,k)=q^{j-2}{\mathscr L}_{{\bf \alpha};q}(n;j-2,k-1)+\sum_{r=0}^s \alpha_r [j-r-1]_q {\mathscr L}_{{\bf \alpha};q}(n;j-r-1,k).
\end{equation}
\end{proposition}
\begin{proof}
Since the proofs are so similar to the ones for the polynomial Stirling numbers, we give only a sketch. To show \eqref{PolyLah}, we specialize \eqref{NOWEYLPOLY} for the word $(Y^2X)^n$ and switch to $\ell=n-k$, 
$$
(Y^2X)^n=\sum_{\ell=1}^{n} \sum_{{\bf k} \in {\mathcal WC}_{s+1}(n-\ell)} m_{{\bf k}}({\mathcal L}_n;q) \, Y^{2n+\sum_{j=0}^s (j-1)k_j} X^{\ell}.
$$
Using the partition \eqref{partition} and letting $j=2n+t$, this yields the assertion. To show \eqref{PolyLahRec}, start from
$$
{\mathscr L}_{{\bf \alpha};q}(n+1;j,k)=\sum_{{\bf k} \in {\mathcal WC}_{s+1}(n-k+1|j-2n-2)} m_{{\bf k}}({\mathcal L}_{n+1};q).
$$
Recall that ${\mathcal L}_{n+1}=C_{2n}\oplus {\mathcal L}_n$. The placements of type $I$ correspond to placements of all rooks on ${\mathcal L}_n$. In $C_{2n}$ there are $2n+(j-2n-2)=j-2$ empty boxes. Writing ${\mathcal WC}_{s+1}(n-k+1|j-2n-2)={\mathcal WC}_{s+1}(n-(k-1)|(j-2)-2n)$, we see that all pacements of type $I$ contribute $q^{j-2}{\mathscr L}_{{\bf \alpha};q}(n;j-2,k-1)$. For a placement of type $II_r$ (with $r=0,\ldots,s$) we place one rook of weight $r$ on $C_{2n}$ and the remaining rooks on ${\mathcal L}_n$. The type ${\bf k}^{(r)}$ of the placement on ${\mathcal L}_n$ is an element of ${\mathcal WC}_{s+1}(n-k|j-2n-2-(r-1))={\mathcal WC}_{s+1}(n-k|(j-r-1)-2n)$. In $C_{2n}$ there are $(j-r-1)$ empty boxes, implying that the contribution of all placements of Type $II_r$ is given by $\alpha_r [j-r-1]_q {\mathscr L}_{{\bf \alpha};q}(n;j-r-1,k)$. 
\end{proof}
\begin{example}
Let $s=1$. In \eqref{PolyLah}, ${\bf k}\equiv (k_0,k_1) \in {\mathcal WC}_{2}(n-k|j-2n)$ means that $k_0+k_1=n-k$ and $-k_0=j-2n$. Thus, the sum reduces to one summand ${\bf k}=(2n-j,j-n-k)$. Hence,
$$
{\mathscr L}_{(\alpha_0,\alpha_1);q}(n;j,k)=m_{2n-j,j-n-k}({\mathcal L}_n;q)=L_{\alpha_0,\alpha_1;q}(n;j,k),
$$ 
recovering the $q$-deformed Ore-Lah numbers $L_{\alpha_0,\alpha_1;q}(n;j,k)$ according to \eqref{QORELAHEX}. Furthermore, the recurrence relation \eqref{PolyLahRec} reduces to the one for the Ore-Lah numbers given in Proposition~\ref{OreLahRec}.
\end{example}
\begin{remark}
In the $q$-deformed Ore algebra we defined the $q$-deformed Ore-Scherk numbers of order $r$, for $r\in \mathbb{N}$, as normal ordering coefficients of $(Y^rX)^n$. In the same fashion, we can define the {\em $q$-deformed polynomial Scherk coefficients of order $r$ and degree $s$}, ${\mathcal S}^{(r)}_{{\bf \alpha};q}(n;j,k)$, as normal ordering coefficients of $(Y^rX)^n$ in  ${\mathcal A}_f(q)$. The same argument as above shows that
\begin{equation}\label{PolyScherk}
{\mathcal S}^{(r)}_{{\bf \alpha};q}(n;j,k)=\sum_{{\bf k} \in {\mathcal WC}_{s+1}(n-k|j-rn)} m_{{\bf k}}(J_{n,r};q).
\end{equation}
For $s=1$, only the summand ${\bf k}=(rn-j,j-(r-1)n-k)$ remains, thus
$$
{\mathcal S}^{(r)}_{(\alpha_0,\alpha_1);q}(n;j,k)=m_{rn-j,j-(r-1)n-k}(J_{n,r};q)=S^{(r)}_{\alpha_0,\alpha_1;q}(n;j,k),
$$
recovering the $q$-deformed Ore-Scherk numbers according to \eqref{OREScherk}. Noting that $J_{n+1,r}=C_{rn}\oplus J_{n,r}$, the same argument as above implies the recurrence relation
$$
{\mathcal S}^{(r)}_{{\bf \alpha};q}(n+1;j,k)=q^{j-r}{\mathcal S}^{(r)}_{{\bf \alpha};q}(n;j-r,k-1)+\sum_{\ell=0}^s \alpha_{\ell} [j-\ell -(r-1)]_q {\mathcal S}^{(r)}_{{\bf \alpha};q}(n;j-\ell -(r-1),k),
$$
which reduces, for $r=1$, to \eqref{RECPOLY} and, for $r=2$, to \eqref{PolyLahRec}. For $s=1$, this gives the recurrence relation of the $q$-deformed Ore-Scherk numbers $S^{(r)}_{\mu,\nu;q}(n;j,k)$.
\end{remark}
Using the same arguments as above for deriving the recurrence relation of the $q$-deformed polynomial Stirling numbers, we can give derive an extension of the result given by Celeste et al. \cite[Theorem 7]{CCG2017} for the $q$-deformed rook numbers $R_{s, h;q}[B_{n}(\lambda),k]$ to the recurrence relation for the $q$-deformed mixed placement numbers. For this, we consider the Ferrers board $B_{\lambda_{n-1}\lambda_{n-2}\cdots \lambda_{1}\lambda_{0}}$, i.e., we label the columns from right to left, and the first column from the right has $\lambda_{0}$ boxes, while the first column to the left has $\lambda_{n-1}$ boxes. If we let $\lambda=\lambda_{n-1}\lambda_{n-2}\cdots \lambda_{1}\lambda_{0}$ be the associated partition, we denote the above board also as $B_n(\lambda)$, and $B_{n-1}(\lambda)=B_{\lambda_{n-2}\cdots  \lambda_{0}}$. 
\begin{theorem}\label{RecurrenceMain}
Let ${\bf k}\equiv (k_0,\ldots,k_s)\in \mathbb{N}_0^{s+1}$. The $q$-deformed mixed placement numbers $m_{{\bf k}}(B_n(\lambda);q)$ satisfy the following recursive formula 
\begin{eqnarray*}
m_{{\bf k}}(B_n(\lambda);q) &=& q^{\lambda_{n-1}+\sum_{j=0}^s (j-1)k_j}m_{{\bf k}}(B_{n-1}(\lambda);q)\\ && +\sum_{r=0}^s \alpha_r \left[\lambda_{n-1}+\sum_{j=0}^s(j-1)k_{j}-(r-1) \right]_q m_{{\bf k}^{(r)}}(B_{n-1}(\lambda);q),
\end{eqnarray*}
where ${\bf k}^{(r)}=(k_0-\delta_{r,0},\ldots,k_{s}-\delta_{r,s})$ and only summands with ${\bf k}^{(r)}\in \mathbb{N}_0^{s+1}$ contribute. For ${\bf k}={\bf 0}$, one has $m_{{\bf 0}}(B_n(\lambda);q)=q^{|\lambda|}$, where $|\lambda|=\lambda_{n-1}+\cdots+\lambda_{0}$ denotes the length of the partition $\lambda$.
\end{theorem}
\begin{proof} 
Exactly as in the proof of Proposition~\ref{THRECPOLY}, we write $B_n(\lambda)=C_{\lambda_{n-1}}\oplus B_{n-1}(\lambda)$ and we partition the set of placement of rooks on $B_n(\lambda)$ into those of type $I$ (all rooks are placed in $B_{n-1}(\lambda)$ and the column $C_{\lambda_{n-1}}$ is left empty), and those of type $II_r$ (for $r=0,\ldots,s$) where one rook of weight $r$ is placed in $C_{\lambda_{n-1}}$ and the remaining rooks are placed in $B_{n-1}(\lambda)$. Since for each placement of type $I$ there are $(\lambda_{n-1}+\sum_{j=0}^s (j-1)k_j)$ empty boxes in $C_{\lambda_{n-1}}$ the contribution of all placements of type $I$ is given by 
$$
q^{\lambda_{n-1}+\sum_{j=0}^s (j-1)k_j}m_{{\bf k}}(B_{n-1}(\lambda);q).
$$
Consider a placement of type $II_r$ (for $r=0,\ldots,s$). For ${\bf k}=(k_0,\ldots,k_s)$, we write as above ${\bf k}^{(r)}\equiv (k_0^{(r)},\ldots,k_{s}^{(r)})=(k_0-\delta_{r,0},\ldots,k_{s}-\delta_{r,s})$. If $\phi' \in \mathcal{M}_{{\bf k}^{(r)}}(B_{n-1}(\lambda))$, then there are $\lambda_{n-1}+\sum_{j=0}^s(j-1)k_{j}^{(r)}= \lambda_{n-1}+\sum_{j=0}^s(j-1)k_{j}-(r-1)$ possible boxes to place the rook of weight $r$ on $C_{\lambda_{n-1}}$ to get a placement $\phi \in \mathcal{M}_{{\bf k}}^{II_r}(B_{n}(\lambda))$. Placing the rook in the $\ell$-th possible row from above (where $\ell=1,\ldots, \lambda_{n-1}+\sum_{j=0}^s(j-1)k_{j}-(r-1))$ will result in $(\lambda_{n-1}+\sum_{j=0}^s(j-1)k_{j}-(r-1)-\ell)$ additional empty boxes. Thus, the sum of the weights of these $\lambda_{n-1}+\sum_{j=0}^s(j-1)k_{j}-(r-1)$ placements $\phi$ is given by $\alpha_r [\lambda_{n-1}+\sum_{j=0}^s(j-1)k_{j}-(r-1)]_q \omega_{{\bf \alpha};q}(B_{n-1}(\lambda),\phi')$. The contribution of all placements of type $II_r$ is, therefore, given by
$$
\alpha_r [\lambda_{n-1}+\sum_{j=0}^s(j-1)k_{j}-(r-1)]_q \, m_{{\bf k}^{(r)}}(B_{n-1}(\lambda);q).
$$
Summing over the types $I$ and $II_r$, for $r=0,\ldots,s$ gives the asserted recurrence. The initial value $m_{{\bf 0}}(B_n(\lambda);q)=q^{|\lambda|}$ follows directly from Definition~\ref{DefMPN}.
\end{proof}
For a pure type ${\bf k}= (0,\ldots,0,k_{\ell},0,\ldots,0)$, one has $
m_{0,\ldots,0,k_{\ell},0,\ldots,0}(B_{n}(\lambda);q)= R_{\ell, \alpha_{\ell};q}[B_{n}(\lambda),k_{\ell}]$, see \eqref{RedPure}. Thus, the recurrence relation of Theorem~\ref{RecurrenceMain} reduces to
\begin{eqnarray*}
R_{\ell, \alpha_{\ell};q}[B_{n}(\lambda),k_{\ell}] &=& q^{\lambda_{n-1}+(\ell-1)k_{\ell}} R_{\ell, \alpha_{\ell};q}[B_{n-1}(\lambda),k_{\ell}] \\ && +\alpha_{\ell}[\lambda_{n-1}+(\ell-1)(k_{\ell}-1)]_q R_{\ell, \alpha_{\ell};q}[B_{n-1}(\lambda),k_{\ell}-1],
\end{eqnarray*}
which is exactly the result given by Celeste et al. \cite[Theorem 7]{CCG2017}. Specializing to $q=1$, one recovers the recurrence relation given by Goldman and Haglund \cite[Equation (3)]{JGJH2000}. 

Similar to the case of the $q$-deformed Ore algebra ${\mathcal O}_{\mu,\nu}(q)$, we can combine \eqref{binomgen2} and \eqref{NOWEYLPOLY} to derive the following generalization of Theorem~\ref{THOREBINOM}
\begin{theorem}\label{THPOLYBINOM} Let $X$ and $Y$ satisfy the commutation relation \eqref{WeylPoly} of the $q$-deformed polynomial Weyl algebra ${\mathcal A}_{f}(q)$. Then one has, for $m\in \mathbb{N}$, the normal ordered binomial formula,
\begin{equation}\label{POLYBINOM}
(X+Y)^m=\sum_{r=1}^{m}\sum_{\ell=r}^m   \sum_{t=r-\ell}^{(s-1)(\ell -r)}  {\mathfrak M}_{m,r,\ell,t}(q) \, Y^{m-\ell+t} X^{r},
\end{equation}
where the combinatorial coefficients ${\mathfrak M}_{m,r,\ell,t}(q)$ are defined by
$$
{\mathfrak M}_{m,r,\ell,t}(q)=\sum_{\lambda \in \mathcal{I}_{m-\ell,\ell}} \sum_{{\bf k} \in {\mathcal WC}_{s+1}(\ell -r|t)} m_{{\bf k}}(B_{\omega_{\lambda}};q). 
$$
\end{theorem}
\begin{proof}
Starting from \eqref{binomgen2}, we observe that the words $\omega_{\lambda}$ have a special structure when compared to the general form considered in Theorem~\ref{THNOWEYLPOLY}. Here all $m_j=1$, for $j=1,\ldots,\ell$, thus $|{\bf m}|=\ell$. Similarly, $|{\bf n}|=|(\lambda_1-\lambda_2,\lambda_2-\lambda_3,\ldots,\lambda_{\ell-1}-\lambda_{\ell},\lambda_{\ell} )|=\lambda_1$. Thus, \eqref{NOWEYLPOLY} reduces to 
$$
\omega_{\lambda} =\sum_{k=0}^{\ell -1} \sum_{{\bf k} \in {\mathcal WC}_{s+1}(k)} m_{{\bf k}}(B_{\omega_{\lambda}};q) Y^{\lambda_1+\sum_{j=0}^s (j-1)k_j} X^{\ell - k}.
$$
Inserting this into \eqref{binomgen2} and switching to $r=\ell -k$ yields
$$
(X+Y)^m=\sum_{\ell=0}^m \sum_{r=1}^{\ell} \sum_{\lambda \in \mathcal{I}_{m-\ell,\ell}}  \sum_{{\bf k} \in {\mathcal WC}_{s+1}(\ell -r)} m_{{\bf k}}(B_{\omega_{\lambda}};q) \, Y^{m-\ell+\sum_{j=0}^s (j-1)k_j} X^{r}.
$$
Switching the order of summation and using the partition \eqref{partition}, one has
$$
(X+Y)^m=\sum_{r=1}^{m}\sum_{\ell=r}^m  \sum_{\lambda \in \mathcal{I}_{m-\ell,\ell}}  \sum_{t=r-\ell}^{(s-1)(\ell -r)}\sum_{{\bf k} \in {\mathcal WC}_{s+1}(\ell -r|t)} m_{{\bf k}}(B_{\omega_{\lambda}};q) \, Y^{m-\ell+\sum_{j=0}^s (j-1)k_j} X^{r}.
$$
For ${\bf k} \in {\mathcal WC}_{s+1}(\ell -r|t)$ one has $\sum_{j=0}^s (j-1)k_j=t$, yielding the assertion.  
\end{proof}
Let $s=1$. Then ${\bf k}\in {\mathcal WC}_{2}(\ell -r|t)$ means that ${\bf k}=(-t,\ell-r-t)$. Since $t$ is negative, we switch to $\tau =-t$ as variable and let also $k\equiv r$. It follows that \eqref{POLYBINOM} reduces to
$$
(X+Y)^m=\sum_{k=1}^{m}\sum_{\ell=k}^m   \sum_{\tau =0}^{\ell -k}  {\mathfrak M}_{m,k,\ell,-\tau}(q) \, Y^{m-\ell-\tau} X^{k},
$$
where the combinatorial coefficients are given by
$$
 {\mathfrak M}_{m,k,\ell,-\tau}(q)=\sum_{\lambda \in \mathcal{I}_{m-\ell,\ell}} m_{\tau,\ell-k-\tau}(B_{\omega_{\lambda}};q).
$$
Thus, we recover Theorem~\ref{THOREBINOM} for the $q$-deformed Ore algebra.

\section{Conclusion}\label{Conclusion}
In this paper, we considered normal ordering in the $q$-deformed Ore algebra. Mixed placements of rooks and files were defined and it was shown that the normal ordering coefficients of arbitrary words are given as mixed placement numbers of rooks and files. The corresponding $q$-deformed Ore-Stirling and Ore-Lah numbers were introduced and studied. Furthermore, we gave an expression for the normal ordered binomial formula in the $q$-deformed Ore algebra. The $q$-deformed Ore algebra can be considered as the case $s=1$ of the $q$-deformed polynomial Weyl algebra (while the case $s=0$ corresponds to the conventional $q$-deformed Weyl algebra). We discussed how the definitions and results can be extended from the $q$-deformed Ore algebra to the $q$-deformed polynomial Weyl algebra. For this, we introduced $q$-deformed mixed placement numbers and showed that the normal ordering coefficients of arbitrary words in the $q$-deformed polynomial algebra are given in terms of these $q$-deformed mixed placement numbers. We introduced and studied the associated $q$-deformed polynomial Stirling and Lah numbers in detail. Since the cases $s=0$ (Weyl algebra) and $s=1$ (Ore algebra) have been studied explicitly, it might be interesting to consider the case $s=2$ -- where $XY-qYX=\alpha_0 I + \alpha_1 Y+\alpha_2Y^2$ -- in more detail, since the relation $XY-qYX =\alpha_2Y^2$ is the defining relation for the $q$-deformed Jordan plane, a well-known object. This will be the subject of further study.

\end{document}